\documentclass[11pt,letterpaper]{amsart}
\usepackage[utf8]{inputenc}
\usepackage[margin=1in]{geometry}
\usepackage{amsmath}
\usepackage{amssymb}
\usepackage{amsthm}
\usepackage{amscd}
\usepackage{enumerate}
\usepackage{xcolor}
\usepackage{subcaption}

\usepackage{graphicx}
\usepackage{amsmath,amsthm,amsfonts,amscd,amssymb,comment,eucal,latexsym,mathrsfs}
\usepackage{stmaryrd}
\usepackage[all]{xy}

\usepackage{epsfig}

\usepackage[all]{xy}
\xyoption{poly}
\usepackage{fancyhdr}
\usepackage{wrapfig}
\usepackage{epsfig}

\theoremstyle{plain}
\newtheorem{thm}{Theorem}[section]
\newtheorem{prop}[thm]{Proposition}
\newtheorem{lem}[thm]{Lemma}
\newtheorem{cor}[thm]{Corollary}

\theoremstyle{definition}
\newtheorem{defn}[thm]{Definition}
\theoremstyle{remark}
\newtheorem{remark}{Remark}
\newtheorem{example}{Example}

  \def\C{{\mathbb{C}}}  \def\E{{\mathbb{E}}} \def\F{{\mathbb{F}}}        \def\N{{\mathbb{N}}}  \def\P{{\mathbb{P}}}  \def\R{{\mathbb{R}}}        

 \def\cB{{\mathcal{B}}}     \def\cG{{\mathcal{G}}} \def\cH{{\mathcal{H}}}     \def\cM{{\mathcal{M}}}  \def\cO{{\mathcal{O}}}   \def\cR{{\mathcal{R}}} \def\cS{{\mathcal{S}}} \def\cT{{\mathcal{T}}}      

\newcommand\dom{\operatorname{dom}}

\newcommand\End{\operatorname{End}}

\newcommand\id{\operatorname{id}}

\newcommand\IRS{\operatorname{IRS}}
\newcommand\dist{\operatorname{d}}

\newcommand\Meas{{\operatorname{Meas}}}

\newcommand\Prob{\operatorname{Prob}}
\newcommand\ran{\operatorname{ran}}

\renewcommand\Re{\operatorname{Re}}

\newcommand\supp{\operatorname{supp}}

\newcommand\Stab{\operatorname{Stab}}
\newcommand\Sub{\operatorname{Sub}}
\newcommand\Span{\operatorname{span}}

\newcommand{\actson}{\curvearrowright}

\newcommand{\ip}[1]{\langle #1 \rangle}

\begin{document}
\title[Co-spectral radius and equivalence relations]{Co-spectral radius for countable equivalence relations}

\author{Mikl\'{o}s Abert}
\address{\parbox{\linewidth}{MTA Alfr\'{e}d R\'{e}nyi Institute of Mathematics
Re\'{a}ltanoda utca 13-15.
H-1053 Budapest,
Hungary
}}
\email{miklos.abert@renyi.mta.hu}
\urladdr{https://www.renyi.hu/~abert/}

\author{Mikolaj Fraczyk}
\address{\parbox{\linewidth}{Faculty of Mathematics and Computer Science, Jagiellonian University, ul. ${\L}$ojasiewicza 6, 30-348 Krak{\'o}w, Poland
}}
\email{mikolaj.fraczyk@uj.edu.pl}
\urladdr{https://sites.google.com/view/mikolaj-fraczyk/home}

\author{Ben Hayes}
\address{\parbox{\linewidth}{Department of Mathematics, University of Virginia, \\
141 Cabell Drive, Kerchof Hall,
P.O. Box 400137
Charlottesville, VA 22904}}
\email{brh5c@virginia.edu}
\urladdr{https://sites.google.com/site/benhayeshomepage/home}

\begin{abstract}
  We define the co-spectral radius of inclusions $\cS\leq \cR$ of discrete, probability measure-preserving equivalence relations, as the sampling exponent of a generating random walk on the ambient relation. The co-spectral radius is analogous to the spectral radius for random walks on $G/H$ for inclusion $H\leq G$ of groups. 
  For the proof, we develop a more general version of the 2-3 method we used in another work on the growth of unimodular random rooted trees. We use this method to show that the walk growth exists for an arbitrary unimodular random rooted graph of bounded degree. We also investigate how the co-spectral radius behaves for hyperfinite relations, and discuss new critical exponents for percolation that can be defined using the co-spectral radius.
  \end{abstract}


\thanks{M. Abert acknowledges support from the KKP 139502 project, the ERC Consolidator Grant 648017 and the Lendulet Groups and Graphs research group. B. Hayes gratefully acknowledges support from the NSF grants DMS-1600802, DMS-1827376, and DMS-2000105. M. Fraczyk was partly supported by the Dioscuri Programme initiated by the Max Planck Society, jointly managed
with the National Science Centre (Poland), and mutually funded by the Polish Ministry of Science and
Higher Education and the German Federal Ministry of Education and Research.}

\maketitle

\section{Introduction}

Let $\Gamma$ be an infinite group generated by a finite symmetric set $S$ and
let $G=\mathrm{Cay}(\Gamma,S)$ be its Cayley graph. We would like to measure
the size of certain subsets $C$ of $\Gamma$, in a translation invariant way, using the graph structure of $G$. As a measuring tool, we will use the lazy\footnote{Laziness is convenient, but it's not necessary. See Section \ref{sec:invariance of almost sure sr} for details.} random walk $(g_n)$ on $G$, starting at the identity, through the sampling probabilities
\[
p_{n,C}=\P(g_{n}\in C).
\]

We are mostly interested in the case when $C$ has zero density and consider the sampling exponent
\[
\rho(C) = \lim_{n\rightarrow\infty}p_{n,C}{}^{1/n}\text{.}%
\]

When $H$ is a subgroup of $\Gamma$, the limit defining $\rho(H)$ will exist and be equal to the spectral radius of the random walk on the quotient Schreier graph $\mathrm{Sch}(\Gamma,H,S)$. Indeed, the covering map $\Gamma\rightarrow\Gamma/H$ gives a bijection between walks
returning to $H$ on $\Gamma$ and walks returning to the root on $\mathrm{Sch}%
(\Gamma,H,S)$. There is considerable literature on this notion \cite{ZukTrees, SoWoess, KaiWoess, BarVir, Woess}, starting with the amenability criterion of Kesten \cite{KestenSR, AVY14, mKesten}.
For arbitrary subsets of $\Gamma$, the sampling exponent will not exist in general, but the picture changes, when it is defined by a $\Gamma$-invariant stochastic process. The most studied such subsets \cite{BenjaminiSchrammConj, PSNPercolation, LyonsSchramm, Hutchcroft, Hutchcroft_Trees,LPBook} are percolation clusters of an i.i.d. (site or bond) percolation on Cayley graphs.

\begin{thm}\label{thm:PercExp}
Let $\Gamma$ be a countable group and consider an i.i.d. percolation on a Cayley graph of $\Gamma$. Then almost surely, for every connected component $C$ of the percolation, the limit \[
\rho(C) = \lim_{n\rightarrow\infty}p_{n,C}{}^{1/n}\text{.}
\]
exists.
\end{thm}

This result is most interesting when the percolation clusters are infinite and there are infinitely many of them a.s. It is easy to see that once $\rho(C)$ exists, it is independent of the starting point of the walk and so, by the indistinguishability theorem of Lyons and Schramm \cite[Theorem 3.3]{LyonsSchramm}, it will be a constant on infinite clusters, depending only on the percolation parameter. Note that when $\Gamma$ is amenable, the above phase does not exist for i.i.d. percolations, but by Kesten's theorem \cite{KestenSR}, in this case, $\rho(C)$ equals $1$ for any subset $C$ anyways. That is, $\rho$ is not a suitable measuring tool for amenable groups. For non-amenable groups, it is a well-known conjecture that the non-uniqueness phase exists for any Cayley graph of the group \cite{BenjaminiSchrammConj}.

We establish the above theorem in a much wider
generality, using the framework of countable measure preserving equivalence
relations. Note that our most general results in this direction (see Theorem \ref{T:existence of the relative spectral radius}, as well as Section \ref{sec:percolation} for the translation between relations and percolation) do not even involve an ambient group anymore, but for the introduction, we stick to group actions.

Let $(X,\mu)$ be a standard Borel probability space and let $\Gamma$ act on
$(X,\mu)$ by $\mu$-preserving maps. We define the \emph{orbit relation}
\[
\cR=\left\{  (x,gx)\in X\times X\mid x\in X,g\in\Gamma\right\}  \text{.}%
\]
A \emph{subrelation} $S$ of $\cR$ is a Borel subset of $\cR$ that, as a relation
on $X$, is an equivalence relation. Orbit relations and their subrelations are
commonly studied objects in measured group theory, see \cite{KM, FurmanSurvey, GaboriauSurvey, Gab1, GroceryList, FurmanMeasRigid}.

Let $\cS$ be a subrelation of $\cR$. For $x,y\in X$ such that $(x,y)\in \cR$
and for
a natural $n$ let the sampling probabilities be
\[
p_{n,x,[y]_{\cS}}=\P((g_{n}x,y)\in \cS)
\]
where, as before, $(g_{n})$ is the lazy random walk on $G$ starting at the identity of $\Gamma$. That is, we walk from $x$ and take the probability that we hit the $\cS$-class of $y$.

\begin{thm}\label{thm:SubExponent}
Let $\Gamma$ be a countable group acting by measure preserving maps on
$(X,\mu)$ and let $\cS$ be a subrelation of the orbit relation of the action.
Then for $\mu$-almost every $y\in X$,
for every $x\in \Gamma y$
\[
\rho^{S}(x,[y]_{\cS}) = \lim_{n\rightarrow\infty}p_{n,x,[y]_{\cS}}{}^{1/n}\text{.}%
\]
exists and is independent of $x$. Moreover, if $\cS$ is either normal or ergodic, then $\rho^{\cS}$ is almost surely constant.
\end{thm}
Note that one can also state an equivalent, stochastic form of Theorem \ref{thm:SubExponent} using the notion of an Invariant Random Partition, that is, a random partition of the group that is invariant in distribution under the shift action. Invariant random partitions are defined in \cite[Section 8.1]{RobinThesis}. They are natural stochastic generalizations of subgroups, as  individually shift invariant partitions are exactly coset partitions with respect to a subgroup. The reformulation says that for any invariant random partition of a countable group, all the partition classes have well defined sampling exponents. We chose to state Theorem \ref{thm:SubExponent} in the relation language because that is the internal language of its proof. We develop the invariant random partition language in the paper \cite{AFHTrajectoryDensity}. 

Theorem \ref{thm:SubExponent} does not seem to follow from the usual arguments. The local environment may look quite different from different points of a class. In algebraic terms, there is no natural quotient object on which the group would act. This is a major deviance from the subgroup case, where this
homogeneity holds and trivially makes $p_{n,H}$ supermultiplicative in $n$. As a result, we were not able to derive the existence of the sampling exponent
with the usual tools, including the standard or Kingman ergodic theorems. For the case of unimodular random trees, Theorem \ref{thm:PercExp} follows from an argument we call the $2$-$3$ method.
In order to prove Theorem \ref{thm:SubExponent} we introduce a generalized version of $2$-$3$ method, which we expect to have more applications. We recall the  rough idea behind the $2$-$3$ method is to establish a local submultiplicative nature of the
sequence and then yield the existence of the limit by a density argument.
The reader can form a quick impression on this method by reading the text after the statement of Theorem \ref{thm: walk growth intro}. 

It turns out that the sampling exponent still admits a spectral interpretation, and equals a norm of a natural Markov-type operator acting on a quotient object of sorts (see Theorem \ref{T:existence of the relative spectral radius}). Because of that and to keep consistency with the subgroup case, we call the sampling exponent $\rho^{S}(x,[y]_{\cS})$ the \emph{co-spectral radius} of the class $[y]_{\cS}$ in the paper. This operator approach also has interesting connections to prior work in percolation theory. For example, our methods recover a lemma due to Schramm on the connectivity decay of random walks in critical percolation which has been used in later results in percolation theory  \cite{KozmaPerc},\cite[Lemma 6.4]{Hutchcroft}, \cite[Section 3]{Hutchcroft_Trees}. We refer the reader to  the discussion preceding Section \ref{sec: 23 method subsection} for more details.

This Markov-type operator is defined on a Hilbert space which is constructed by integrating the bundle of Hilbert spaces $\ell^{2}([x]_{\cR}/\cS)$ into one Hilbert space. While we will not need it for this work, this Hilbert space can be naturally related to the Jones' basic construction of the inclusion $L(\cS)\leq L(\cR)$ of von Neumann algebras of the corresponding equivalence relations, as well as the $L^{2}$-space of a natural measure space occurring in \cite{FSZ} (see Definition 1.4 of that paper).
We expect that our new methods will have applications to the operator algebraic setting. 
We refer the reader to the discussion preceding Section \ref{subsec:proof of average version}. 

We now state the 2-3 method theorem in its most general form that leads to Theorem \ref{thm:SubExponent}. 

\begin{thm}\label{thm:generalized 23INTRO}
Let $\cR$ be a discrete, measure-preserving equivalence relation over a standard probability space $(X,\mu)$, and fix $\pi\in L^{1}(X,\mu)$ with $\pi(x)\in (0,\infty)$ for almost every $x\in X$.
Let $f_{k}\colon \cR\to [0,\infty]$ for $k\geq 1$ be a sequence of  measurable functions such that:
\begin{enumerate}[(a)]
\item $f_{k}$ is $\pi$-symmetric for all $k\in \N$, i.e. $f_{k}(x,y)\pi(y)=f_{k}(y,x)\pi(x)$ for almost every $(x,y)\in \cR$,
\item for all $l,k\in \N$ we have $\sum_{y,z\in [x]_{\cR}}f_{l}(x,z)f_{k}(z,y)\leq \sum_{y\in [x]_{\cR}}f_{l+k}(x,y)$ for almost every $x\in X$, \label{item:subadditive convolution inequality intro}
\item for almost every $x\in X$ and every $k\in \N$ we have $0<\sum_{y\in [x]_{\cR}}f_{k}(x,y)<\infty$, \label{item:almost sure convergence intro}
\item there is a measurable $D\colon X\to (0,\infty)$ so that  $\sum_{y\in [x]_{\cR}}f_{l+k}(x,y)\geq D(x)^{l}\sum_{y\in [x]_{\cR}}f_{k}(x,y)$ for almost every $x\in X$ and every $l,k\in \N$. \label{item:small log scale errors intro}
\end{enumerate}
Then
\[\widetilde{f}(x)=\lim_{k\to\infty}\left(\sum_{y\in [x]_{\cR}}f_{k}(x,y)\right)^{1/k}\]
exists and is positive almost surely. Further:
\begin{enumerate}[(i)]
\item for every $k\in \N$  \[\int\pi(x)\frac{\sum_{y\in [x]_{\cR}}f_{k}(x,y)}{\widetilde{f}(x)^{k}}\,d\mu(x)\leq \int \pi\,d\mu,\] \label{item: average monotonicity intro}
\item $\lim_{k\to\infty}\left(\int \pi(x)\sum_{y\in [x]_{\cR}}f_{k}(x,y)\,d\mu(x)\right)^{1/k}=\|\widetilde{f}\|_{\infty}.$ \label{item:integral formula for sup intro}
\end{enumerate}
\end{thm}

In the special case when $f_k(x,y)$ is the probability of transition from $x$ to $y$ by a standard random walk  in $k$ steps, the first part of the theorem can be seen as a large deviations estimate, in the sense that it controls the density of starting points where the random walk sampling probability deviates from what is suggested by the co-spectral radius. In fact, it is natural to ask whether our main result holds in the ``annealed" sense, that is, when we take expected value of the sampling probabilities before we take the n-th root. A warning comment here is that in this case we have to consider the event of returning to the class of the starting point, because equivalence classes often can not be individually identified in a measurable way (this is what is called indistinguishability in percolation theory, which is equivalent to the ergodicity of a subrelation in the measured language).
In any case, this ``annealed" version is much simpler to prove than our main result and most of the effort in the paper is spent to establish the pointwise (or ``quenched") version.
Additionally, we show that the ``annealed" version is the essential supremum of the ``quenched" version and that in many cases, the ``quenched version" is a.s. constant.  See Theorem \ref{T:existence of the relative spectral radius} for a precise relation between the ``annealed" and ``quenched" versions. As a sample application, in the case of Bernoulli bond percolation it follows from the indistinguishability result of Lyons-Schramm \cite{LyonsSchramm} and our work that the ``annealed" version and the ``quenched" version agree on infinite clusters.

\begin{remark}
The following was pointed out to us by the anonymous referee. As in Theorem \ref{thm:PercExp} consider an i.i.d percolation on the Cayley graph of a group $\Gamma$. Let $X_{n}$ be a random walk on $\Gamma$ with $X_{0}=e$ and with transition probabilities $\P(X_{n}=a|X_{n-1}=b)=\nu(b^{-1}a)$ for some $\nu\in \Prob(\Gamma)$ whose support generates $\Gamma$. We let $\P(X_{0}\leftrightarrow X_{n})$ be the probability that $X_{n}$ is in the connected component of $e$ in this percolation, and we let $C$ be the connected component of $e$ in this percolation. 
Theorem \ref{thm:generalized 23INTRO} implies in the context of Theorem \ref{thm:PercExp} that
\[\lim_{n\to\infty}\P(X_{0}\leftrightarrow X_{n})^{1/n}=\lim_{n\to\infty}\P(X_{0}\leftrightarrow X_{n}|C)^{1/n}\]
almost surely, where $C$ is the cluster of $X_{0}$. The subadditive ergodic theorem shows for, e.g., Bernoulli percolation that almost surely
\[\lim_{n\to\infty} \P(X_0 \leftrightarrow X_n | X_n )^{1/n} = \exp \left[ \lim_{n\to\infty} \frac{1}{n} \E \log \P(X_0 \leftrightarrow X_n) \right].\]
So in order to get the correct almost-sure decay of the connection probability conditioned on the random walk, one must take the limit of the expectation of the logarithm. Our results says that if one instead conditions on the cluster $C$, then one does not need to take the expectation of the logarithm.
In this sense, our result may be thought of as saying that the major contribution to $\E[\P(X_{0} \leftrightarrow X_{n}|C)]$ is from contributions which are of ``typical size" (up to subexponential factors). For the expectation $\E[\P(X_{0}\leftrightarrow X_{n}|X_{n})]$ this is not the case, and often rare events contribute substantially to the expectation.
\end{remark}

\subsection{The $2$--$3$--method and walk growth}
We now present another application of the general $2$-$3$ (see \cite[Section 3]{AFHGrowth} for another application), showing the existence of the exponential rate of growth of the number of walks in any unimodular random rooted graph (see also Example \ref{ex-23Walks} in Section \ref{sec: 23 method subsection}). We also sketch the proof, as it illustrates quite well what goes into the 2-3 method that is behind Theorem \ref{thm:SubExponent}.
Then we state the most general version of the 2-3 method.

\begin{thm}\label{thm: walk growth intro}
Let $(\cG,o)$ be a connected bounded degree unimodular random rooted graph with degree at most $d$. Let $w_n(o)$ denote the number of length $n$ walks starting at $o$. Then the limit $\lim_{n\to\infty} \frac{1}{n}\log w_n(o)$ exists. Further, if $(\cG,o)$ is ergodic and $\eta\in \Prob(\cM_{d})$ is the distribution of $(\cG,o)$ define $A\in B(L^{2}(\cM_{d},\eta))$ by
\[(Af)([(G,o)])=\sum_{w\thicksim o}f(w).\]
Then $A$ is a self-adjoint operator and $\log\|A\|=\lim_{n\to\infty}\frac{1}{n}\log w_{n}(o).$
\end{thm}


The usual technique used to establish the rate of growth in ergodic theory is the Kingman sub-additive theorem. We weren't able to find any action or equivalence relation with a sub-multiplicative cocycle that would control the number of walks in $\cG$, so we couldn't use it to solve the problem. As $w_n(o)$ can be naturally expressed as an inner product, one is also tempted to use spectral theory, but this also did not work for us. Instead we use the mass transport principle to show that the inequalities
\[w_{2n}(o)\gg w_n(o)^2 \textnormal{ and } w_{3n}(o)\gg w_n(o)^3\] hold with overwhelming probability, as $n$ gets large. To see why this is useful, imagine that we know that these inequalities hold always and with the implicit constant $1$. Then, the sequence $w_{2^p3^q}(o)$ is sup-multiplicative, so the limit
\[ \lim_{p,q\to\infty} \frac{1}{2^p3^q}\log w_{2^p3^q}(o)\] exists. The function $n\mapsto \log w_n(o)$ is Lipschitz and the set $2^p3^q$ is dense on the logarithmic scale, so we can deduce that the limit $\lim_{n\to\infty}\frac{1}{n}\log w_n(o)$ also exists.

\subsection{Outline of the paper}
Section \ref{sec:Notation} introduces the necessary background for the paper, including a discussion of measure-preserving equivalence relations and how to reduce the study of percolation clusters to equivalence relations. Section \ref{sec:co spectral radius setup} contains a proof of the co-spectral radius and its basic properties. In that section we also include some background on representations of equivalence relations, so that in  subsection \ref{subsec:proof of average version} we may identify the co-spectral radius with an operator norm. In Section \ref{sec: 23 method subsection} we give a proof of the general $2$-$3$ method which gives us the pointwise existence of the co-spectral radius as a special case. 
In Section \ref{sec:invariance of almost sure sr} we show that the co-spectral radius is almost surely constant if the subrelation is ergodic or normal. In Section \ref{sec: hyperfinite relations}, we show that the co-spectral radius agrees with the spectral radius when the subrelation is hyperfinite and give a counterexample to the converse (i.e. a Kesten's theorem for subrelations) using monotone couplings of IRS's. In Section \ref{sec:percolation} we use the co-spectral radius to define new critical exponents for percolation. Finally, in Section \ref{sec: walk growth} we use the $2$-$3$ method to establish the existence of walk growth and relate it to an operator norm.

\hfill 

\noindent \textbf{Remark.} Note that this paper and parts of \cite{AFHGrowth} and \cite{AFHTrajectoryDensity} first appeared on arXiv as one long text.  Following explicit suggestions of helpful referees, we decided to separate the work into the three papers, also to make the results more accessible to their natural audiences. 

\hfill

\noindent \textbf{Acknowledgements.} Much of the work by the last named author was done on visits to the R\'{e}nyi institute in Budapest. He would like to thank the R\'{e}nyi Institute for its hospitality. We would like to thank Gabor Pete for helpful conversations.  We thank the anonymous referee for their numerous comments, which greatly improved the paper.

\section{Background and Notation}\label{sec:Notation}
A \emph{standard probability space} is  pair $(X,\mu)$ where $X$ is a standard Borel space, and $\mu$ is the completion of a Borel probability measure on $X$. We say that $E\subseteq X$ is \emph{measurable} if it is in the domain of $\mu$.
An \emph{equivalence relation over $(X,\mu)$} is a Borel subset $\cR\subseteq X\times X$ so that the relation $\thicksim$ on $X$ given by $x\thicksim y$ if $(x,y)\in \cR$ is an equivalence relation. For $x\in X$, we let $[x]_{\cR}=\{y\in X:(x,y)\in \cR$. We say that $\cR$ is \emph{discrete} if for almost every $x\in X$ we have that $[x]_{\cR}$ is countable. If $\cR$ is discrete, we may turn $\cR$ into a $\sigma$-finite measure space by endowing $\cR$ with the Borel measure
\[\overline{\mu}(E)=\int_{X}|[x]_{\cR}\cap E|\,d\mu(x)\mbox{ for all Borel $E\subseteq \cR$.}\]
We will continue to use $\overline{\mu}$ for the completion of $\overline{\mu}$. If $\cR$ is discrete, we say that it is \emph{measure-preserving} if the map $\cR\to \cR$ given by $(x,y)\mapsto (y,x)$ is measure-preserving. Equivalently, this just means that the \emph{mass-transport principle} holds: if $f\colon \cR\to [0,\infty]$ is Borel, then
\[\int_{X}\sum_{y\in [x]_{\cR}}f(x,y)\,d\mu(x)=\int_{X}\sum_{y\in [x]_{\cR}}f(y,x)\,d\mu(x).\]
For a group $\Gamma$, and $S\subseteq \Gamma$, we use $\ip{S}$ for the subgroup of $\Gamma$ generated by  $S$.
If $\Gamma$ is a countable group, and $\Gamma\actson (X,\mu)$ is a measure-preserving action, then $\cR_{\Gamma,X}=\{(x,gx):g\in \Gamma\}$ is a discrete, measure-preserving equivalence relation.
We use the notation $\cS\leq \cR$ to mean that $\cS$ is a \emph{subequivalence relation} of $\cR$, namely a subset of $\cR$ which is also an equivalence relation over $(X,\mu)$.
We often abuse terminology and say that $\cS$ is a \emph{subrelation} of $\cR$, and leave it as implicit that $\cS$ should also be an equivalence relation.
If  $(X,\mu)$ and $E\subseteq X$ is measurable we let
\[\cR|_{E}=\cR\cap (E\times E).\]
If $E$ has positive measure, then $\cR|_{E}$ is a measure-preserving relation over the probability space $(E,\frac{\mu(E\cap \cdot)}{\mu(E)})$. We let $[\cR]$ be the group of all bimeasurable bijections $\phi\colon X\to X$ so that $\phi(x)\in [x]_{\cR}$ for almost every $x\in X$. We identify two elements of $[\cR]$ if they agree almost everywhere. We have a natural metric $d$ on $[\cR]$ given by
\[d(\phi,\psi)=\mu(\{x\in X:\phi(x)\ne \psi(x)\}).\]
This is a complete, separable, translation-invariant metric on $[\cR]$ and this turns $[\cR]$ into a \emph{Polish group}. We use $\Prob([\cR])$ for the Borel probability measures on $[\cR]$. Since $[\cR]$ is a Polish group, the space $\Prob([\cR])$ can be made into a semigroup under convolution: so if $\nu_{1},\nu_{2}\in \Prob([\cR])$, then $\nu_{1}*\nu_{2}\in \Prob([\cR])$ is defined by
\[(\nu_{1}*\nu_{2})(\Omega)=\nu_{1}\otimes \nu_{2}(\{(\phi,\psi):\phi\psi\in \Omega\}).\]
Given a countable $\Gamma\leq [\cR]$, we say that $\Gamma$ \emph{generates} $\cR$ if $\Gamma x=[x]_{\cR}$ for almost every $x\in X$. Given a countably supported $\nu\in \Prob([\cR])$, we say that $\nu$ \emph{generates $\cR$} if $\ip{\supp(\nu)}$ generates $\cR$, where $\supp(\nu)=\{\phi \in [\cR]:\nu(\phi)\ne 0\}$. We say that $\nu\in \Prob([\cR])$ is \emph{symmetric} if the map $\phi\mapsto \phi^{-1}$ preserves $\nu$.

Let $G$ be a graph. We write $V(G), E(G)$ for the vertex and the edge set of $G$. Let $v\in V(G)$ and $r\in\mathbb N$. The $r$-ball around $v$ is denoted by $B_G(v,r)$ and the $r$-sphere by $S_G(v,r)$.

We use Vinogradov's notation and write $f\ll g$ if $|f|$ is bounded by a constant times $|g|$.

For a Banach space $V$, we let $B(V)$ be the space of continuous, linear operators $T\colon V\to V$. For $T\in B(V)$, we set:
\[\|T\|=\sup_{v\in V:\|v\|\leq 1}\|T(v)\|.\]
At various times we will have to appeal to spectral theory of bounded self-adjoint operators on a Hilbert space. Since most standard references on this theory assume Hilbert spaces are complex, in order to make these applications most transparent all Hilbert spaces will be assumed complex throughout the paper.

\subsection{Translation between percolation and equivalence relations}\label{sec:translation}
Some of our results are stated in the percolation theory language but all our proofs will be based on measured equivalence relations and graphings. An invariant (edge) percolation on a unimodular random graph $(\mathcal G,o)$ is a random triple $(\mathcal G,o,P)$ where $P$ is a subset of edges of $G$ and the distribution of the triple is invariant under the re-rooting equivalence relation. The following proposition associates a p.m.p\footnote{probability measure preserving} measured equivalence relations to a percolation is such a way that Theorem \ref{thm:PercExp} can deduced from Theorem \ref{thm:SubExponent}. 
\begin{prop}\label{prop:PercGraph} Let $(\mathcal G,o)$ be a unimodular random graph with an invariant percolation $P$. There exists a p.m.p countable equivalence relation $(\Omega_\#,\nu_\#,\cR)$ with a generating graphing $(\varphi_i)_{i\in I}$ and a sub-relation $\cS\subset \cR$ such that
\begin{enumerate}
    \item The rooted graph $(\cG_\omega,\omega)$ with the vertex set $[\omega]_\cR$ and the edge set $\{(\omega', \varphi_i(\omega')): \omega'\in [\omega]_{\cR}, i\in I\}$ has the same law as $(\cG,o)$.
    \item The law of the pairs $(P^o,\cG)$ where $P^o$ is the connected component of the percolation $P\subset \mathcal G$ is the same as the law of $([\omega]_\cS, \cG_\omega)$ where $[\omega]_{\cS}\subset \cG_\omega$ is the $\cS$ equivalence class of $\omega$.
\end{enumerate}
\end{prop}
\begin{proof}
 We follow closely the construction in \cite[Example 9.9]{AldousLyons}. Let $\Omega$ be the space of pairs $((G,o),S)$ where $(G,o)$ is a rooted graph of degree at most $d$ and $S$ is a subset of edges. The distribution of the percolation $P$ is naturally a probability measure on $\Omega$, let us call it $\mu$. Let $\Omega_\#$ be a the set of triples $((G,o),S,\lambda),$ where $(G,o),S$ are as before and $\lambda$ is a two coloring of the vertices of $(G,o)$. Let $\mu_\#$ be the distribution of the random triple $((\cG,o),P,\Lambda)$, where $\Lambda$ is an i.i.d coloring. The space $(\Omega_\#,\mu_\#)$ is equipped with a natural finite graphing $\Phi$ in which $((G,o),S,\lambda),((G',o'),S',\lambda')$ are connected if and only if $G=G', S=S', \lambda=\lambda'$ and $o'$ is a neighbour of $o$. The graphing $\Phi$ spans the re-rooting measured equivalence relation $\cR$, which preserves $\mu_\#$. For each point $\omega\in \Omega_\#$, the equivalence class $[\omega]_\cR$ is equipped with a bounded degree graph structure $\cG_\omega$. The resulting random rooted graph $(\cG_\omega,\omega)$ has the same law as $(\mathcal G,o)$ and the second coordinate has the same law as the percolation $P$. To construct the sub-relation $\cS$, we select a sub-graphing $\Phi'\subset \Phi$ where $((G,o),S,\lambda),((G',o'),S',\lambda')$ are connected if and only if $G=G', S=S', \lambda=\lambda'$ and $o'$ is a neighbour of $o$ connected by an edge in $S$. In this way the connected component of $P$ containing the root is given by the $\cS$-equivalence class of $w$ in $\cG_w$.
\end{proof}

\section{Existence of the co-spectral radius for subrelations}\label{sec:co spectral radius setup}

Let $\cR$ be an ergodic probability measure-preserving equivalence relation over a standard probability space $(X,\mu)$ and let $\nu\in \Prob([\cR])$ be countably supported and symmetric, i.e. $\nu(\{\phi\})=\nu(\{\phi^{-1}\})$ for all $\phi\in [\cR]$.  Consider a measurable subequivalence relation $\cS\leq \cR$. For $x\in X$, the measure $\nu$ determines a random walk on $[x]_{\cR}$ with transition probabilities $p_{y,z}=\nu(\{\phi:\phi(y)=z\})$.
For $n\in \N$, $(x,y)\in \cR$, we let $p_{n,x,y}^{\nu}$ (resp. $p_{n,x,\cS}^{\nu}$) be the probability that the random walk corresponding to $\nu$ starting at $x$ is at $y$ after $n$ steps (resp. in $[x]_{\cS}$ after $n$ steps). By direct calculation,
\[p^{\nu}_{n,x,y}=\nu^{*n}(\{\phi:\phi(x)=y\}).\]
If $\nu$ is clear from the context (which is the usually the case), we will use $p_{n,x,\cS},p_{n,x,x}$ instead of $p_{n,x,\cS}^{\nu},p_{n,x,x}^{\nu}.$ We are interested in the existence of the co-spectral radius of $\cS$ inside $\cR$ which is, by definition, the limit
\[\lim_{n\to\infty}p_{2n,x,\cS}^{\frac{1}{2n}}.\]
In particular, we will show that this limit exists almost surely. While there are easy examples where this limit genuinely depends upon $x$ (see Example \ref{example: direct sum counterexample}) we will show that in many cases it is almost surely constant and is the norm of a self-adjoint operator on a Hilbert space naturally associated to $\cS\leq \cR$.

This is, of course, motivated by the case of an inclusion of groups $\Delta\leq \Gamma.$ Here  the existence of the co-spectral radius, as well as the fact that it is the norm of the corresponding Markov operator on $\ell^{2}(\Gamma/\Delta)$ is a nontrivial, but well known, fact. In the case $\Gamma$ is finitely-generated and $\nu$ is the uniform measure on a finite generating set of $\Gamma$ we are looking at a random walk on a  Schreier graph and it is easier to see existence of this limit using the natural action of $\Gamma$ on $\Gamma/\Delta$. Even in the case $\nu$ is not the uniform measure on a finite subset of $\Gamma$ the action $\Gamma\actson \Gamma/\Delta$ naturally enters into the very definition of the Markov operator. In the relation case this a priori presents a problem.

Because $\cS$ is a subrelation of $\cR$, for $x\in X$ we can divide $[x]_{\cR}$ into $\cS$-equivalence class. We let $[x]_{\cR}/\cS$ be the space of $\cS$-equivalence classes in $[x]_{\cR}$. The field of spaces $[x]_{\cR}/\cS$ is analogous to $\Gamma/\Delta$, and so we may consider $\ell^{2}([x]_{\cR}/\cS)$ as analogous to $\ell^{2}(\Gamma/\Delta)$. However, there is no obvious natural action of $[\cR]$ on $[x]_{\cR}/\cS$ and this makes it difficult to see how one would define a Markov operator, and thus the co-spectral radius. We proceed to explain how to navigate this difficulty by collecting the field of Hilbert spaces $\ell^{2}([x]_{\cR}/S)$ together in a natural object.

\begin{defn}Let  $(X,\mu)$ be a standard probability space, then a \emph{measurable field of Hilbert spaces over $X$} is a family $(\mathcal{H}_{x})_{x\in X}$ of separable Hilbert spaces, together with a family $\Meas(\mathcal{H}_{x})\subseteq \prod_{x\in X}\mathcal{H}_{x}$  so that:
\begin{itemize}
\item for every $(\xi_{x})_{x},(\eta_{x})_{x}\in \Meas(\mathcal{H}_{x})$ we have that $x\mapsto \ip{\xi_{x},\eta_{x}}$ is measurable,
\item if $\eta=(\eta_{x})_{x\in X}\in \prod_{x\in X}\mathcal{H}_{x}$ and $x\mapsto \ip{\xi_{x},\eta_{x}}$ is measurable for all $\xi=(\xi_{x})_{x}\in \Meas(\mathcal{H}_{x}),$ then $\eta\in \Meas(\mathcal{H}_{x}),$
\item there is a sequence $(\xi^{(n)})_{n=1}^{\infty}$ with $\xi^{(n)}=(\xi^{(n)})_{x\in X}$ in $\Meas(\mathcal{H}_{x})$ so that $\mathcal{H}_{x}=\overline{\Span\{\xi^{(n)}_{x}:n\in \N\}}$ for almost every $x\in X.$
\end{itemize}
The direct integral, denoted $\int_{X}^{\oplus} \mathcal{H}_{x}\,d\mu(x),$ is defined to be all $\xi\in \Meas(\mathcal{H}_{x})$ so that $\int_{X}\|\xi_{x}\|^{2}\,d\mu(x)<\infty,$ where we identify two elements of $\Meas(\mathcal{H}_{x})$ if they agree outside a set of measure zero. We put an inner product on $\int_{X}^{\oplus}\mathcal{H}_{x}$ by
\[\ip{\xi,\eta}=\int_{X}\ip{\xi_{x},\eta_{x}}\,d\mu(x),\]
and this gives $\int_{X}^{\oplus}\mathcal{H}_{x}$ the structure of a Hilbert space.
\end{defn}
We shall typically drop ``over $X$" in ``a measurable field of Hilbert space over $X$" if $X$ is clear from the context.
For later use, if $(\mathcal{H}_{x})_{x}$ is a measurable field of Hilbert spaces, then given $A\subseteq X$ measurable and $\xi\in \Meas(\mathcal{H}_{x}),$ we let $1_{A}\xi\in \Meas(\mathcal{H}_{x})$ be defined by
\[(1_{A}\xi)_{x}=1_{A}(x)\xi_{x}.\]

In our case, we can give the family $(\ell^{2}([x]_{\cR}/\cS))_{x}$ a measurable structure by declaring that $\xi=(\xi_{x})_{x}\in \prod_{x\in X}\ell^{2}([x]_{\cR}/\cS)$ is measurable if $x \mapsto \xi_{x}([\phi(x)]_{\cS})$ is measurable, for all $\phi\in [\cR]$. General facts about direct integral imply this collection of measurable vectors satisfy the above axioms (see Lemma \ref{lem:Hilbert space structure}). So we can define $L^{2}(\cR/\cS)$ by
\[L^{2}(\cR/\cS)=\int_{X}^{\bigoplus}\ell^{2}([x]_{\cR}/\cS)\,d\mu(x).\]
As mentioned above, there is no obvious natural action of $[\cR]$ on $[x]_{\cR}/\cS$. However, we do have a natural unitary representation  of $[\cR]$ on $L^{2}(\cR/\cS)$. Define
\[\lambda_{\cS}\colon [\cR]\to \mathcal{U}(L^{2}(\cR/\cS))\]
by
\[(\lambda_{\cS}(\phi)\xi)_{x}=\xi_{\phi^{-1}(x)}.\]
We will not need it for this paper, but this can be regarded as a representation of $\cR$ itself (a precise definition will be given in \cite{AFHTrajectoryDensity}). For our purposes, we simply note that we have natural Markov operators defined on $L^{2}(\cR/\cS)$. Namely, for a countably supported $\nu\in \Prob([\cR])$, we define \[\lambda_{\cS}(\nu)=\sum_{\phi\in [\cR]}\nu(\phi)\lambda(\phi).\]
Here we are mildly abusing notation and using $\nu(\phi)$ for $\nu(\{\phi\})$, this will not present problems since $\nu$ is atomic.

\begin{thm}\label{T:existence of the relative spectral radius}
Let $\cR$ be a measure-preserving equivalence with countable orbits over a standard probability space $(X,\mu)$, and let $\nu\in \Prob([\cR])$ be atomic. Suppose that the support of $\nu$ generates $\cR$. Fix $\cS\leq \cR$.
\begin{enumerate}[(i)]
\item \label{I:averageversion subrelation} The limit
\[\rho(\cR/\cS,\nu):=\lim_{n\to\infty}\left(\int p_{2n,x,\mathcal{S}}\,d\mu(x)\right)^{\frac{1}{2n}}\]
exists. Moreover,
\[\rho(\cR/\cS,\nu)=\|\lambda_{\cS}(\nu)\|.\]
\item \label{item:almost sure cs radius} The pointwise limit
\[\rho_{\nu}^{\cS}(x)=\lim_{n\to\infty}p_{2n,x,\cS}^{1/2n}\]
exists almost surely, and
\[\rho(\cR/\cS,\nu)=\|\rho_{\nu}^{\cS}\|_{\infty}.\]
\item \label{item: ergodic normalizer}
Suppose that the partial one-sided normalizer of $\cS\leq \cR$ acts ergodically (see Definition \ref{defn:1-sided normalizer} for the definition). Then $\rho_{\nu}^{\cS}$ is almost surely constant, and by (\ref{item:almost sure cs radius}) equals $\rho(\cR/\cS,\nu)$.  In particular, this applies if $\cS$ is normal or ergodic.
\end{enumerate}
\end{thm}

We will often drop the $\nu$ from $\rho^{\cS}_{\nu}$ if it is clear from context, and simply write $\rho^{\cS}$. Let $X,\mu,\cR$ be as in Theorem \ref{T:existence of the relative spectral radius}. Suppose that $y\in X$ and that the limit defining $\rho^{\cS}(y)$ exists. Given $x\in [y]_{\cR}$, choose a $k\in\N$ with $p_{2k,x,y}>0$. Then,
\[p_{2k,x,y}p_{2(n-k),y,[y]_{\cS}}\leq p_{2n,x,[y]_{\cS}}\leq p_{2(n+k),y,[y]_{\cS}}p_{2k,x,y}^{-1},\]
and so the limit
\[\lim_{n\to\infty}p_{2n,x,[y]_{\cS}}^{1/2n}\]
exists and equals $\rho^{\cS}(y)$. If $\nu$ is assumed lazy, then
\[p_{2k,x,y}p_{y,y}p_{2(n-k),y,[y]_\cS}\leq p_{2n+1,x,[y]_{\cS}}\leq p_{2(n+k+1),y,[y]_{\cS}}p_{2k,x,y}^{-1}p_{y,y}^{-1},\]
and so
\[\lim_{n\to\infty}p_{n,x,[y]_{\cS}}^{1/n}\]
exists and equals $\rho^{\cS}(x)$. Thus Theorem \ref{T:existence of the relative spectral radius} recovers Theorem \ref{thm:SubExponent}.

\begin{remark}
Co-spectral radius for normal subrelations also occurs in \cite[Lemma 6.7]{BHAExtension}, however in that context the subrelation is both normal and ergodic, which gives a well-defined quotient group. If the subrelation is normal, there is a quotient groupoid \cite{FSZ} however our situation is general enough (encompassing  when the subrelation is ergodic \emph{or} when it is normal) that we cannot appeal directly to the group case as in \cite{BHAExtension}. We remark that the space $L^{2}(\cR/\cS)$ is closely related to the relation $\widehat{\cS}$ that appears in \cite[Definition 1.4]{FSZ} (we caution the reader the roles of $\cS,\cR$ are reversed in \cite{FSZ} relative to our work and so this relation is denoted $\widehat{\cR}$ there), which is a measure-preserving relation on the space $Y=\{(x,c):x\in X,c\in [x]_{\cR}/\cS\}$.

Explicitly,
\[\widehat{\cS}=\{(x,c),(y,\widetilde{c}):(x,y)\in \cR,c=\widetilde{c}\}.\]
We proceed to explain how they are related. Since we not explicitly use the connection between $L^{2}(\cR/\cS)$ and $L^{2}(\widehat{\cS})$, we will only sketch the details.

In \cite{AFHTrajectoryDensity}, we will explain how to give $Y$ the structure of a standard Borel space and equip it with a natural $\sigma$-finite measure. This measure will be defined in such a way to make $L^{2}(Y)$ naturally unitarily isomorphic to $L^{2}(\cR/\cS)$. We also have an isometric embedding $V$ of $L^{2}(\cR/\cS)$ into $L^{2}(\widehat{\cS})$ given by
\[(Vf)(x,c,y,c)=\delta_{x=y}f(x,c).\]
Part of the significance of the relation $\widehat{\cS}$ for the results in \cite{FSZ} is that certain properties of the inclusion $\cS\leq \cR$ (e.g. the index, normality) are reflected in terms of properties of the inclusion $\cR\leq \widehat{\cS}$. An alternative explanation for this can be given by von Neumann algebras: let $L(\cS),L(\cR)$ be the von Neumann algebras of the equivalence relations $\cS,\cR$ as defined in \cite{FelMoore} (the analogous notation there is $M(\cS),M(\cR)$). We then have a natural inclusion of von Neumann algebras $L(\cS)\leq L(\cR)$. The von Neumann algebra $L(\widehat{\cS})$ can be realized as the \emph{basic construction} $\widehat{M}=\ip{L(R),e_{L(S)}}$, in the sense of Jones \cite[Section 3]{Jones83},  of $L(\cS)\leq L(\cR)$. For the interested reader, we remark that under this correspondence, the space $L^{2}(\cR/\cS)$ corresponds to the following subspace of $L^{2}(\widehat{M})$
\[\cH=\overline{\Span\{fu_{\phi}e_{L(\cS)}u_{\phi}^{-1}:f\in L^{\infty}(X,\mu),\phi\in [\cR]\}},\]
where $u_{\phi}$ are the canonical unitaries in $L(\cR)$ corresponding to the elements of $[\cR]$ and $e_{L(\cS)}$ is the Jones projection corresponding to the inclusion $L(\cS)\leq L(\cR)$ (see \cite{FSZ}). Moreover, the action of $[\cR]$ naturally acts on $\cH$ by conjugating by $u_{\phi}$, and this action is isomorphic to the action of $[\cR]$ on $L^{2}(\cR/\cS)$ we define above. We refer to \cite{Jones83}, \cite[Appendix F]{BO}, for the appropriate definitions, which we will not need in this work.    
\end{remark}

We now proceed to prove (\ref{I:averageversion subrelation}) of the above Theorem, whose proof is almost entirely operator theory.

\subsection{Proof of Theorem \ref{T:existence of the relative spectral radius} (\ref{I:averageversion subrelation})}\label{subsec:proof of average version}

The essential idea behind Theorem \ref{T:existence of the relative spectral radius} (\ref{I:averageversion subrelation}) is that we have a natural vector in $L^{2}(\cR/\cS)$ which is given by the measurable field $\xi_{x}=\delta_{[x]_{\cS}}.$ By a direct calculation, we have that
\[\ip{\lambda_{\cS}(\nu)^{2n}\xi,\xi}=\int p_{2n,x,\cS}\,d\mu(x).\]
We then have to show that
\[\|\lambda_{\cS}(\nu)\|=\lim_{n\to\infty}\ip{\lambda_{\cS}(\nu)^{2n}\xi,\xi}^{\frac{1}{2n}}.\]
Because $\nu$ is symmetric, the operator $\lambda_{\cS}(\nu)$ is self-adjoint and the existence of the limit on the right-hand side follows from the spectral theorem \cite[Theorems IX.2.2 and IX.2.3]{Conway}. The limit on the right-hand side is also dominated by $\|\lambda_{\cS}(\nu)\|$. We prove a general Hilbert space theorem which, after a small amount of work, will prove the reverse inequality. First, let us alleviate any concerns about the measurable structure defined on $(\ell^{2}([x]_{\cR}/\cS))_{x\in X}$.

\begin{lem}\label{lem:Hilbert space structure}
Let $\cS\leq \cR$ be discrete measure-preserving equivalence relations over a standard probability space $(X,\mu)$. Define a family $\Meas(\ell^{2}([x]_{\cR}/\cS))\subseteq \prod_{x\in X}\ell^{2}([x]_{\cR}/\cS)$ by saying that $(\xi_{x})_{x\in X}\in \Meas(\ell^{2}([x]_{\cR}/\cS))$ if and only if for every $\phi \in [\cR]$ the function $X\to \C$ given by $x\mapsto \xi_{x}([\phi(x)]_{\cS})$ is measurable. Then the family $\Meas(\ell^{2}([x]_{\cR}/\cS))$ turns $(\ell^{2}([x]_{\cR}/\cS))_{x}$ into a measurable field of Hilbert spaces.
\end{lem}

\begin{proof}
Fix a countable subgroup $\Gamma\leq [\cR]$ so that $\Gamma x=[x]_{\cR}$ for almost every $x$. We start by proving the following claim.

\emph{Claim: a vector $\xi=(\xi_{x})_{x}\in \prod_{x\in X}\ell^{2}([x]_{\cR}/\cS)$ is in $\Meas(\ell^{2}([x]_{\cR}/\cS))$ if and only if the map $x\mapsto \xi_{x}([\phi(x)]_{\cS})$ is measurable for every $\phi \in \Gamma$.}
If $\xi$ is measurable then, by definition, $x\mapsto \xi_{x}([\phi(x)]_{\cS})$ is measurable for every $\phi\in[\cR]$. In particular, this is true for $\phi\in\Gamma$. Conversely, suppose that $x\mapsto \xi_{x}([\phi(x)]_{\cS})$ is measurable for all $\phi \in\Gamma.$ Fix a $\psi\in [\cR]$. We then have to show that $x\mapsto \xi_{x}([\psi(x)]_{\cS})$ is measurable. Since $\Gamma x=[x]_{\cR}$ for almost every $x\in X$, we may find a disjoint family of sets $(E_{\phi})_{\phi\in\ \Gamma}$ with  $E_{\phi}\subseteq \{x\in X:\psi(x)=\phi(x)\}$ and so that $\bigsqcup E_{\phi}$ is a conull subset of $X$. Then
\[\xi_{x}([\psi(x)]_{\cS})=\sum_{\phi \in\Gamma}1_{E_{\phi}}(x)\xi_{x}([\phi(x)]_{\cS})\]
for almost every $x$ (the sum above converges since the $E_{\phi}$ are disjoint). Since $\Gamma$ is countable, this proves that $x\mapsto \xi_{x}([\psi(x)]_{\cS})$ is measurable, and this proves the claim.

Having shown the claim, for $\phi\in \Gamma$ define $\zeta_{\phi}\in \prod_{x\in X}\ell^{2}([x]_{\cR}/\cS)$ by $\zeta_{\phi,x}=\delta_{[\phi(x)]_{\cS}}$. Then
\[\ell^{2}([x]_{\cR}/\cS)=\overline{\Span\{\zeta_{\phi,x}:\phi \in\Gamma\}}^{\|\cdot\|_{2}}\]
for almost every $x\in X$, and also
$\xi_{x}([\phi(x)]_{\cS})=\ip{\xi_{x},\delta_{\phi(x)}}$
for all $\phi\in [\cR]$. Moreover, for $\phi,\psi\in\Gamma$ we have that
\[\ip{\zeta_{\phi,x},\zeta_{\psi,x}}=1_{S}(\phi(x),\psi(x))\]
which is a measurable function of $x$.
The lemma now follows from countability of $\Gamma$ and \cite[Lemma IV.8.10]{TakesakiI}.
\end{proof}

We use the following well known Lemma (see e.g. \cite[Equation 2.8]{HutchcroftIsing} for a proof), which is the main way we will relate the operator norm of the Markov operator $\lambda_{\cS}(\nu)$ to the growth of the matrix coefficients $\ip{\lambda_{\cS}(\nu)^{2n}\xi,\xi}$.  

\begin{lem}\label{L:norm on dense set of vectors}
Let $\mathcal{H}$ be a Hilbert space and $T\in B(\mathcal{H})$ self-adjoint. Let  $K$ be an index set, and let $(\xi_{k})_{k\in K}$ be a $K$-tuple of vectors in $\mathcal{H}$ so that $\mathcal{H}=\overline{\Span\{\xi_{k}:k\in K\}}.$ Then
\[\|T\|=\sup_{k\in K}\lim_{n\to\infty}\ip{T^{2n}\xi_{k},\xi_{k}}^{1/2n}.\]

\end{lem}

In order to apply this in the context of a direct integral of Hilbert spaces, the following density criterion will be useful.

\begin{lem}\label{L:dense set from graphing}
Let $(X,\mu)$ be a standard probability space and let $(\mathcal{H}_{x})_{x\in X}$ be a measurable family of Hilbert spaces over $X$ and set $\mathcal{H}=\int_{X}^{\bigoplus}\mathcal{H}_{x}.$ Suppose we have a sequence $\xi_{n}=(\xi_{n,x})_{x\in X}\in \Meas(\mathcal{H}_{x})$ such that
\[\mathcal{H}_{x}=\overline{\Span\{\xi_{n,x}:n\in \N\}}\]
for almost every $x\in X$. Then
\[\mathcal{H}=\overline{\Span\{1_{A}\xi_{n}:n\in \N, A\subseteq X\mbox{ is measurable}\}}.\]
\end{lem}

\begin{proof}
Suppose that $\eta\in \cH$ and that $\ip{\eta,1_{A}\xi_{n}}=0$ for every $n\in \N$ and every measurable $A\subseteq X$. Then for every measurable $A\subseteq X$ and every $n\in \N$ we have
\[\int_{A}\ip{\eta_{x},\xi_{n,x}}\,d\mu(x)=0.\]
Since this holds for every $A$,
applying this with $A$ being 
$\{x\in X:\Re(i^{j}\ip{\eta_{x},\xi_{n,x}})>0\}$ 
for $j=0,1,2,3$, and taking real and imaginary parts of the above integral shows that
that for every $n\in \N$ we have that $\ip{\eta_{x},\xi_{n,x}}=0$ for almost every $x\in X$. 
By countability, for almost every $x\in X$ we have $\ip{\eta_{x},\xi_{n,x}}=0$ for all $n\in \N$. Since $\mathcal{H}_{x}=\overline{\Span\{\xi_{n,x}:n\in \N\}}$ for almost every $x\in X$, we deduce that $\eta_{x}=0$ for almost every $x\in X$, i.e. $\eta=0$ as an element of $\cH$. Thus we have shown that the only vector in $\cH$ orthogonal to $\{1_{A}\xi_{n}:n\in \N, A\subseteq X\mbox{ is measurable}\}$ is the zero vector, and this implies that
\[\mathcal{H}=\overline{\Span\{1_{A}\xi_{n}:n\in \N, A\subseteq X\mbox{ is measurable}\}}.\]
\end{proof}

\begin{proof}[Proof of Theorem \ref{T:existence of the relative spectral radius} (\ref{I:averageversion subrelation})]
Let $\xi\in L^{2}(\cR/\cS)$ be the measurable vector field given by $\xi_{x}=\delta_{[x]_{\cS}}.$ By direct calculation,
\[\ip{\lambda_{S}(\nu)^{2n}\xi,\xi}=\int p_{2n,x,\cS}\,d\mu(x).\]
By the spectral theorem, there is a probability measure $\eta$ on $[-\|\lambda_{S}(\nu)\|,\|\lambda_{S}(\nu)\|]$ so that
\[\ip{\lambda_{S}(\nu)^{2n}\xi,\xi}=\int t^{2n}\,d\eta(t).\]
From this, we see that $\lim_{n\to\infty}\ip{\lambda_{S}(\nu)^{2n}\xi,\xi}^{\frac{1}{2n}}$ exists and is the $L^{\infty}$-norm of $t$ with respect to $\eta$. Combining these results we see that
\[\lim_{n\to\infty}\left(\int p_{2n,x,\cS}\,d\mu(x)\right)^{1/2n}\]
exists. Call this limit $\rho(\cR/\cS,\nu)$ as in the statement of the Theorem.

We now turn to the proof that $\rho(\cR/\cS,\nu)=\|\lambda_{\cS}(\nu)\|$. It follows from the logic in the preceding paragraph that $\rho(\cR/\cS,\nu)\leq \|\lambda_{\cS}(\nu)\|$. Let $\Gamma$ be the subgroup of $[\cR]$ generated by the support of $\nu$. Since $\nu$ generates $\cR$, we have $[x]_{\cR}=\Gamma x$ for almost every $x\in X$. By Lemma \ref{L:dense set from graphing},
\[\{1_{A}\lambda_{S}(\phi)\xi:\phi\in \Gamma, A\subseteq X\mbox{ is measurable}\}\]
has dense linear span in $\cH$. It thus suffices, by Lemma \ref{L:norm on dense set of vectors}, to prove that \[\lim_{n\to\infty}\ip{\lambda_{\cS}(\nu)^{2n}1_{A}\lambda_{\cS}(\phi)\xi,1_{A}\lambda_{\cS}(\phi)\xi}^{\frac{1}{2n}}\leq \rho(\cR/\cS,\nu)\]
for every measurable $A\subseteq X$ and $\phi\in \Gamma$. It is direct to see that
\[\ip{\lambda_{\cS}(\nu)^{2n}1_{A}\lambda_{\cS}(\phi)\xi,1_{A}\lambda_{\cS}(\phi)\xi}\leq \ip{\lambda_{\cS}(\nu)^{2n}\lambda_{\cS}(\phi)\xi,\lambda_{\cS}(\phi)\xi}\]
for every measurable $A\subseteq X$. So it simply suffices to show that
\begin{equation}\label{eqn: reduction to generating graphing}\lim_{n\to\infty}\ip{\lambda_{\cS}(\nu)^{2n}\lambda_{\cS}(\phi)\xi,\lambda_{\cS}(\phi)\xi}^{\frac{1}{2n}}\leq \rho(\cR/\cS,\nu),\textnormal{ for every $\phi\in\Gamma.$}
\end{equation}

To prove (\ref{eqn: reduction to generating graphing}), fix $\phi \in \Gamma.$ Let $\phi^{*}(\cS)$ be the subrelation $\phi^{*}(\cS)=\{(\phi(x),\phi(y)):(x,y)\in \cS\}$. By direct computation,
\[\ip{\lambda_{\cS}(\nu)^{2n}\lambda_{\cS}(\phi)\xi,\lambda_{\cS}(\phi)\xi}=\int p_{2n,\phi(x),\phi^{*}(\cS)}\,d\mu(x).\]
Choose $k\in \N$ so that $c=\nu^{*k}(\{\phi\})>0$. Then for every $n\in \N$ we have that
\begin{align*}p_{2(n+k),x,\cS}&=\sum_{y\in [x]_{\cS}}p_{2(n+k),x,y}\geq \sum_{y\in [x]_{\cS}}p_{k,x,\phi(x)}p_{2n,\phi(x),\phi(y)}p_{k,\phi(y),y}\\ &\geq c^{2}\sum_{y\in [x]_{\cS}}p_{2n,\phi(x),\phi(y)}=c^{2}p_{2n,\phi(x),\phi^{*}(\cS)}.
\end{align*}
Integrating both sides we obtain
\[\ip{\lambda_{\cS}(\nu)^{2n}\lambda_{\cS}(\phi)\xi,\lambda_{\cS}(\phi)\xi}\leq c^{-2}\ip{\lambda_{\cS}(\nu)^{2(n+k)}\xi,\xi}.\]
Thus,
\begin{align*}\lim_{n\to\infty}\ip{\lambda_{\cS}(\nu)^{2n}\lambda_{\cS}(\phi)\xi,\lambda_{\cS}(\phi)\xi}^{\frac{1}{2n}}&\leq\lim_{n\to\infty}c^{-1/n}\left[\ip{\lambda_{\cS}(\nu)^{2(n+k)}\xi,\xi}^{\frac{1}{2(n+k)}}\right]^{\frac{n+k}{n}}\\&=\rho(\cR/\cS,\nu).\end{align*}
This proves (\ref{eqn: reduction to generating graphing}), and so completes the proof of Theorem \ref{T:existence of the relative spectral radius} (\ref{I:averageversion subrelation}).

\end{proof}

As mentioned in the introduction, Theorem \ref{T:existence of the relative spectral radius} and our later work in \ref{sec: hyperfinite relations} can be used to
recover a result due to Schramm used in
\cite{KozmaPerc},\cite[Lemma 6.4]{Hutchcroft}, \cite[Section 3]{Hutchcroft_Trees}. Indeed, give a percolation of  a connected, regular, transitive graph $G$ by Section \ref{sec:translation} one can build an inclusion $\cS\leq \cR$ of relations on a standard probability space as well as a symmetric probability measure $\nu$ on $[\cR]$ so that in the notation of \cite[Lemma 6.4]{Hutchcroft}, we have
\[\int p_{n,x,\cS}\,d\mu(x)=\E[\tau_{p}(X_{0},X_{n})].\]
For $p\in (0,1)$, consider Bernoulli(p) edge percolation, where each edge is kept with probability $p$, let $\cS_{p}\leq \cR_{p}$ be the corresponding inclusion of equivalence relations. Set
\[p_{c}=\inf\{p: \mbox{a.e. connected component in Bernoulli(p) percolation is finite}\}.\]
In the setup of the proof of Theorem \ref{T:existence of the relative spectral radius} (\ref{I:averageversion subrelation}), we have that
\[\int p_{2n,x,\cS}\,d\mu(x)=\ip{\lambda_{\cS}(\nu)^{2n}\xi,\xi}.\]
Since $\lambda_{\cS}(\nu)$ is self-adjoint, the spectral theorem tells us that
\[\left(\int p_{2n,x,\cS}\,d\mu(x)\right)^{1/2n}=\ip{\lambda_{\cS}(\nu)^{2n}\xi,\xi}^{1/2n}\]
is increasing in $n$ and Theorem \ref{T:existence of the relative spectral radius} (\ref{I:averageversion subrelation}) characterizes its supremum as $\|\lambda_{\cS}(\nu)\|$. Additionally, for every $n\in \N$ we by Cauchy-Schwarz that
\[\left(\int p_{2n,x,\cS}\,d\mu(x)\right)^{1/n}=\ip{\lambda_{\cS}(\nu)^{n}\xi,\xi}^{1/n}\leq \|\lambda_{\cS}(\nu)^{n}\xi\|^{1/n}=\ip{\lambda_{\cS}(\nu)^{2n}\xi,\xi}^{1/2n},\]
where in the last step we use self-adjointness of $\lambda_{\cS}(\nu)$.
The construction of the inclusion $\cS\leq \cR$ forces that the spectral radius of $G$ is equal to the co-spectral radius of $\cT\leq \cR$ where $\cT=\{(x,x):x\in X\}$ is the trivial relation. For Bernoulli(p) percolation with $p<p_{c}$, finiteness of the clusters tells us that $[x]_{\cS_{p}}$ is finite for almost every $x\in X$. We will later show (see Proposition \ref{prop:hyperfintie doesn't change sr}) that this implies that
$\rho(\cR/\cS_{p},\nu)=\|\lambda_{\cS_{p}}(\nu)\|\leq \|\lambda_{\cT}(\nu)\|=\rho(\cR,\nu)$. Using the standard monotone coupling of Bernoulli percolation \cite[Section 5.2]{LPBook}, it is direct to see that $\|\lambda_{\cS_{p}}(\nu)\|$ is semicontinuous and that $\|\lambda_{\cS_{p_{c}}}(\nu)\|\leq \|\lambda_{\cT}(\nu)\|$.  Thus we obtain the estimate of Schramm for connected,  transitive graphs. With minor modifications, our results work for random walks on finite cost graphings (see Example \ref{example: finite cost} of Section \ref{sec: 23 method subsection}). In this manner, we can recover the same estimate of Schramm when $G$ is a connected, locally finite, transitive graph. As mentioned in the introduction, in the context of percolation all our proofs can be rephrased without equivalence relations and can be done in the language of percolation theory.

\subsection{The 2--3 method and Proof of Theorem \ref{T:existence of the relative spectral radius} (\ref{item:almost sure cs radius})}\label{sec: 23 method subsection}

We now explain how to deduce the existence of the  pointwise limit defining the co-spectral radius. We first state the general Theorem behind this existence and then explain why it applies to our setting, as well as to more general situations. For notation, if $f,g\colon \cR\to [0,\infty]$ are measurable, then we define their \emph{convolution} to be the function $f*g\colon \cR\to [0,\infty]$ given by
\[(f*g)(x,y)=\sum_{z\in [x]_{\cR}}f(x,z)g(z,y).\]
Given a measurable $\pi\colon X\to \C$, we say that $f\colon \cR\to [0,\infty]$ is \emph{$\pi$-symmetric} if $\pi(x)f(x,y)=\pi(y)f(y,x)$ for almost every $(x,y)\in \cR$. If $\pi=1$, we just say $f$ is \emph{symmetric}.

\begin{thm}\label{thm:generalized 23}
Let $\cR$ be a discrete, measure-preserving equivalence relation over a standard probability space $(X,\mu)$, and fix $\pi\in L^{1}(X,\mu)$ with $\pi(x)\in (0,\infty)$ for almost every $x\in X$.
Let $f_{k}\colon \cR\to [0,\infty]$ for $k\geq 1$ be a sequence of  measurable functions such that:
\begin{enumerate}[(a)]
\item $f_{k}$ is $\pi$-symmetric for all $k\in \N$,
\item for all $l,k\in \N$ we have $\sum_{y\in [x]_{\cR}}(f_{l}*f_{k})(x,y)\leq \sum_{y\in [x]_{\cR}}f_{l+k}(x,y)$ for almost every $x\in X$, \label{item:subadditive convolution inequality}
\item for almost every $x\in X$ and every $k\in \N$ we have $0<\sum_{y\in [x]_{\cR}}f_{k}(x,y)<\infty$, \label{item:almost sure convergence}
\item there is a measurable $D\colon X\to (0,\infty)$ so that  $\sum_{y\in [x]_{\cR}}f_{l+k}(x,y)\geq D(x)^{l}\sum_{y\in [x]_{\cR}}f_{k}(x,y)$ for almost every $x\in X$ and every $l,k\in \N$. \label{item:small log scale errors}
\end{enumerate}
Then
\[\widetilde{f}(x)=\lim_{k\to\infty}\left(\sum_{y\in [x]_{\cR}}f_{k}(x,y)\right)^{1/k}\]
exists and is positive almost surely. Further:
\begin{enumerate}[(i)]
\item for every $k\in \N$  \[\int\pi(x)\frac{\sum_{y\in [x]_{\cR}}f_{k}(x,y)}{\widetilde{f}(x)^{k}}\,d\mu(x)\leq \int \pi\,d\mu,\] \label{item: average monotonicity}
\item $\lim_{k\to\infty}\left(\int \pi(x)\sum_{y\in [x]_{\cR}}f_{k}(x,y)\,d\mu(x)\right)^{1/k}=\|\widetilde{f}\|_{\infty}.$ \label{item:integral formula for sup}
\end{enumerate}
\end{thm}

The name 2--3 refers to the way we are proving Theorem \ref{thm:generalized 23}. In the proof we have two separate steps where we show that ``typically" $\sum_{y}f_{2k}(x,y)\sim (\sum_y f_k(x,y))^2$ and $(\sum_y f_{3k}(x,y))\sim (\sum_y f_k(x,y))^3$. Since $2,3$ generate a multiplicative semi-group which is asymptotically dense on the logarithmic scale, we are able to deduce that the exponential growth of $(\sum_y f_k(x,y))$ has a definite rate.

Before jumping into the proof of Theorem \ref{thm:generalized 23}, let us list several examples where it applies. For all of these examples, fix a discrete, measure-preserving equivalence relation $\cR$ over a standard probability space $(X,\mu)$ and fix a $\cS\leq \cR$.

\begin{example}
Fix a symmetric $\nu\in \Prob([\cR])$.
Define
\[f_{k}(x,y)=p_{2k,x,y}1_{\cS}(x,y).\]
It is direct to check that our hypotheses apply in this case with $\pi=1$, $D(x)=p_{2,x,x}$.
In this case
\[\sum_{y\in [x]_{\cS}}f_{k}(x,y)=p_{2k,x,\cS},\]
and we recover the existence of the pointwise co-spectral radius $\rho^{\cS}_{\nu}$. Moreover, item (\ref{item:integral formula for sup}) of Theorem \ref{thm:generalized 23} as well as Theorem \ref{T:existence of the relative spectral radius} (\ref{I:averageversion subrelation}) imply that
\[\|\lambda_{\cS}(\nu)\|=\rho(\cR/\cS,\nu)=\|\rho^{\cS}_{\nu}\|_{\infty}.\]
So we recover the operator norm of the Markov operator $\lambda_{\cS}(\nu)$ as the $L^{\infty}$-norm of $\rho^{\cS}_{\nu}$.
\end{example}

\begin{example}\label{example: finite cost}
Suppose that $\nu\colon \cR\to [0,1]$ is measurable and that $\sum_{y\in [x]_{\cR}}\nu(x,y)=1$ for almost every $x\in X$. Moreover, assume that there is a $\pi\colon X\to (0,\infty)$ with $\pi\in L^{1}(X,\mu)$ so that $\nu$ is $\pi$-symmetric.  Consider the Markov chain on $[x]_{\cR}$ with transition probabilities $\nu(x,y)$, and for $k\in \N$ and $(x,y)\in \cR$, let $p_{k,x,y}$ be the probability that the random walk corresponding to this Markov chain starting at $x$ is at $y$ after $k$ steps. Set
\[f_{k}(x,y)=p_{2k,x,y}1_{S}(x,y),\]
and $p_{2k,x,\cS}=\sum_{y\in [x]_{\cS}}p_{2k,x,y}$. Note that if $f,g\colon \cR\to [0,\infty]$ are $\pi$-symmetric, then so is $f*g$. Since $f_{k}=\nu^{*2k}1_{S},$ we see that $f_{k}$ is $\pi$-symmetric. Finally, observe that for every $x\in X,k\in \N$
\[0<p_{2,x,x}^{k}\leq p_{2k,x,x}\leq \sum_{y\in [x]_{\cR}}f_{k}(x,y)\leq 1.\]
So we deduce the existence of
\[\rho(\cR/\cS,\nu)=\lim_{k\to\infty}p_{2k,x,\cS}^{\frac{1}{2k}}.\]

A good example to keep in mind is the following.  Recall that the \emph{full pseudogroup}, denoted $[[\cR]]$, of $\cR$ is by definition the set of bimeasurable bijections $\phi\colon \dom(\phi)\to \ran(\phi)$ satisfying
\begin{itemize}
    \item $\dom(\phi),\ran(\phi)$ are measurable subsets of $X$,
    \item $\phi(x)\in [x]_{\cR}$ for almost every $x\in X$.
\end{itemize}
We identify $\phi,\psi\in [[\cR]]$ if $\mu(\dom(\phi)\Delta \dom(\psi))=0$ and if $\phi(x)=\psi(x)$ for almost every $x\in \dom(\phi)\cap \dom(\psi)$. For $\phi,\psi\in [[\cR]]$ we use $\phi\circ \psi$ for the element of $[[\cR]]$ whose domain is $\dom(\psi)\cap \psi^{-1}(\dom(\phi))$ and which satisfies $\phi\circ \psi(x)=\phi(\psi(x))$ for $x\in \dom(\psi)\cap \psi^{-1}(\dom(\phi))$. For $\phi\in [[\cR]]$, we let $\phi^{-1}$ be the element of $[[\cR]]$ with $\dom(\phi^{-1})=\ran(\phi)$, $\ran(\phi^{-1})=\dom(\phi)$, and so that $\phi^{-1}(\phi(x))=x$ for all $x\in \dom(\phi)$.
Given a countable $\Phi\subseteq [[\cR]]$ and $x\in X$ we define a graph $G_{\Phi,x}$ whose vertex set is $[x]_{\cR}$ and whose edge set is $\bigcup_{\phi \in \Phi}\{\{y,\phi^{\alpha}(y)\}:\alpha\in \{\pm 1\},y\in [x]_{\cR}\cap \dom(\phi^{\alpha})\}$. We say that $\Phi$ is a \emph{graphing} if $G_{\Phi,x}$ is connected for almost every $x\in X$. We define the \emph{cost} of a graphing to be
\[c(\Phi)=\sum_{\phi\in \Phi}\mu(\dom(\phi))\]
 This definition is due to  Levitt in \cite{Levitt} and the cost of a relation (which is by defining the infimum of the cost of its graphings) was further systematically studied in \cite{Gab1, GroceryList}. If $\Phi$ is a finite cost graphing, then for almost every $x\in X$ define $\nu(x,y)=\frac{1}{\deg_{G_{\Phi}}(x)}$ where $\deg_{G_{\Phi}}(x)$ is the degree of $x$ in the graph $G_{\Phi,x}$. Observe that $\nu$ is $\deg_{G,\Phi}$-symmetric. Since $\int \deg_{G_{\Phi}}\,d\mu\leq 2c(\Phi)<\infty$, we have a well-defined co-spectral radius for the simple random walk associated to finite cost graphings.

 Another good example is the following. Consider a symmetric $\nu\in \Prob([\cR])$, and $E\subseteq X$ a measurable set with $\mu(E)>0$. Assume that for almost every $x\in E$, we have that $\sum_{y\in [x]_{\cR}\cap E}p_{x,y}>0$ (e.g. this holds if the random walk is lazy). 
 
 For $x\in X$, define $\nu|_{E}\colon \cR\cap (E\times E)\to [0,1]$ by
 \[\nu|_{E}(x,y)=\frac{p_{x,y}}{\sum_{y\in [x]_{\cR}\cap E}p_{x,y}}.\]
 As above this defines a Markov chain on $[x]_{\cR}\cap E$ with transition probabilities given by $\nu|_{E}$. We have that $\nu|_{E}$ is $\pi$-symmetric, with $\pi(x,y)=\sum_{y\in [x]_{\cR}\cap E}p_{x,y}$. For $k\in \N$, we let $p^{\nu|_{E}}_{k,x,y}$ be the probability that the random walk starting at $x$ with these transition probabilities is at $y$ after $k$ steps.  Setting $p_{2k,x,\cS|_{E}}^{\nu|_{E}}=\sum_{y\in [x]_{\cS}\cap E}p^{\nu|_{E}}_{2k, x,y}$, we deduce the almost sure existence of the ``conditional" co-spectral radius
\[\rho^{\cS|_{E}}(x)=\lim_{k\to\infty}(p_{2k,x,\cS|_{E}}^{\nu|_{E}})^{\frac{1}{2k}}\]
for $x\in E$.
\end{example}

\begin{example}\label{example:local cs}
Fix a symmetric $\nu\in \Prob([\cR])$ and a measurable $E\subseteq X$ with $\mu(E)>0.$ Define
\[f_{k}(x,y)=p_{2k,x,y}1_{\cS}(x,y)1_{E}(x)1_{E}(y).\]
In this case for $x\in E$ we have
\[\sum_{y\in [x]_{\cS}}f_{k}(x,y)=p_{2k,x,\cS,E}\]
where $p_{2k,x,\cS,E}$ is the probability that the random walk corresponding to $\nu$ starting at $x$ is in $[x]_{\cS}\cap E$ after $2k$ steps. All of our hypotheses apply in this case with $X$ replaced with $E$, $\cR$ replaced with $\cR|_{E}$, and $\pi=1$.
So we recover the existence of the \emph{pointwise local co-spectral radius}
\[\rho_{E}^{\cS}(x)=\lim_{k\to\infty}p_{2k,x,\cS,E}^{1/2k},\]
 at least for $x\in E$. See Section \ref{sec:invariance of almost sure sr} for more details. In that section we will show more generally that
 \[\lim_{k\to\infty}p_{2k,x,\cS,E}^{1/2k}\]
 exists for almost every $x\in X$, see Corollary \ref{cor: local 23}. This specific example will be important for us when we show that the spectral radius is almost surely constant in the case that $\cS$ is either normal or ergodic. \end{example}

 We say that a $\nu\in \Prob([\cR])$ is \emph{lazy} if $\nu(\{\id\})>0$.

\begin{example}
Fix a lazy, symmetric $\nu\in \Prob([\cR])$,  and  a measurable $E\subseteq X$ with $\mu(E)>0.$ Define for $(x,y)\in \cR$:
\[p_{2k,x,y,E}=\sum_{x_{1},\cdots,x_{2k-1}\in E}p_{x,x_{1}}p_{x_{1},x_{2}}\cdots p_{x_{2k-1},y},\]
\[p_{2k,x,\cS|_{E}}=\sum_{\substack{x_{1},\cdots,x_{2k}\in E,\\ (x,x_{2k})\in \cS}}p_{x,x_{1}}p_{x_{1},x_{2}}\cdots p_{x_{2k-1},x_{2k}}.\]
Define
\[f_{k}(x,y)=p_{2k,x,y,E}1_{E}(x)1_{E}(y)1_{S}(x,y).\]
Then
\[\sum_{y\in [x]_{\cS}}f_{k}(x,y)=p_{2k,x,\cS|_{E}}.\]
Again, it is direct to check that the hypotheses of Theorem \ref{thm:generalized 23} apply with $\pi=1$. 
Note that laziness implies that 
\[p_{2k,x,\cS|_{E}}\geq p(x,x)^{2k}>0\]
for every $x\in E$ and we always have 
\[p_{2k,x,\cS|_{E}}\leq 1.\]
Thus (\ref{item:almost sure convergence}) of Theorem \ref{thm:generalized 23} holds.
So we deduce the existence of the ``restricted" co-spectral radius
\[\rho_{\cS|_{E}}(x)=\lim_{k\to\infty}p_{2k,x,\cS|_{E}}^{1/2k}\]
for $x\in E$.
It can be shown by the same method of proof as Theorem \ref{T:existence of the relative spectral radius} (\ref{I:averageversion subrelation}) that
\[\lim_{k\to\infty}\left(\int p_{2k,x,\cS|_{E}}\,d\mu(x)\right)^{\frac{1}{2k}}=\|1_{E}\lambda_{\cS}(\nu)1_{E}\|,\]
so
\[\|\rho_{\cS|_{E}}\|_{\infty}=\|1_{E}\lambda_{\cS}1_{E}\|.\]
So we again recover the norm of a corner of the Markov operator as the essential supremum of the restricted co-spectral radius.
\end{example}

\begin{example}\label{ex-23Walks}
Let $\varepsilon>0$ and let $\eta\colon \cR\to [\varepsilon,+\infty)$ be a measurable bounded symmetric function. Define for $(x,y)\in \cR$:
\[f_k(x,y):=\sum_{x_{1},\cdots,x_{2k-1}\in [x]_\cR}\eta_{x,x_{1}}\eta_{x_{1},x_{2}}\cdots\eta_{x_{2k-1},y}.\]
Theorem \ref{thm:generalized 23} yields almost sure existence of the growth exponent $\rho_\eta:=\lim_{k\to\infty} \left(\sum_{y\in [x]_\cR} f_k(x,y)\right)^{1/k}.$ Let $\Phi\subset \cR$ be a symmetric graphing generating $\cR$.
If we put $\eta=1_{\Phi}$, we get the existence of the growth exponent for the number of length $2k$ trajectories starting at $x$ in the graph induced by $\Phi$ on the equivalence class of $[x]_\cR$. In particular, for every unimodular random graph $(\cG,o)$, the number of walks of length $2k$ starting from $o$ has an exponential rate of growth almost surely. Note that such a result is definitely not true for any rooted bounded degree graph. This last example is discussed in more detail in Section \ref{sec: walk growth}.
\end{example}

Having explained why Theorem \ref{thm:generalized 23} implies the existence of the pointwise co-spectral radius, we now turn to the proof of Theorem \ref{thm:generalized 23}. The following is the main technical lemma behind the proof.

\begin{lem}\label{lem:weak submultiplicativity}
Let $(f_{k})_{k},\pi$ be as in Theorem \ref{thm:generalized 23}. For $k\in \N$, define $\widetilde{f}_{k}\colon X\to [0,\infty]$ by
\[\widetilde{f}_{k}(x)=\sum_{y\in [x]_{\cR}}f_{k}(x,y),\]
and define
\[\phi_{k}(x)=\sum_{y\in [x]_{\cR}}f_{k}(x,y)\widetilde{f}_{k}(y)^{-1},\]
\[\psi_{k}(x)=\widetilde{f}_{k}(x)\sum_{y\in [x]_{\cR}}f_{k}(x,y)\left(\sum_{z\in [y]_{\cR}}f_{k}(y,z)\widetilde{f}_{k}(z)\right)^{-1}.\]
Then:
\begin{enumerate}[(i)]
\item for every $k\in \N$ and almost every $x\in X$,\label{item:CS estimate}
\[\widetilde{f}_{k}(x)^{2}\leq \phi_{k}(x)\widetilde{f}_{2k}(x)\mbox{ and } \widetilde{f}_{k}(x)^{3}\leq \psi_{k}(x)\widetilde{f}_{3k}(x),\]
\item for every $k\in \N$: \label{item:mass transport 23}
\[\int \pi\phi_{k}\,d\mu=\int \pi\,d\mu=\int \pi\psi_{k}\,d\mu.\]
\end{enumerate}

\end{lem}

Note that in order to make sense of $\phi_{k},\psi_{k}$ we are using hypothesis (\ref{item:almost sure convergence}).

\begin{proof}

(\ref{item:CS estimate})
By Cauchy-Schwarz,
\begin{align*}
  \widetilde{f}_{k}(x)^{2}=\left(\sum_{y\in [x]_{\cR}}f_{k}(x,y)^{1/2}\widetilde{f}_{k}(y)^{-1/2}f_{k}(x,y)^{1/2}\widetilde{f}_{k}(y)^{1/2}\right)^{2}&\leq \phi_{k}(x)\sum_{y\in [x]_{\cR}}f_{k}(x,y)\widetilde{f}_{k}(y)\\
  &\leq \phi_{k}(x)\widetilde{f}_{2k}(x)  
\end{align*}
for almost every $x\in X$, where in the last step we use the hypothesis (\ref{item:subadditive convolution inequality}).
Using Cauchy-Schwarz again:
\[\widetilde{f}_{k}(x)^{2}\leq \widetilde{f}_{k}(x)^{-1}\psi_{k}(x)\sum_{y\in [x]_{\cR}}f_{k}(x,y)\left(\sum_{z\in [y]_{\cR}}f_{k}(y,z)\widetilde{f}_{k}(z)\right) \mbox{ for a.e. $x\in X$.}\]
We can estimate the second sum using our hypothesis on $f_{k}$:
\begin{align*}
\sum_{y\in [x]_{\cR}}f_{k}(x,y)\left(\sum_{z\in [y]_{\cR}}f_{k}(y,z)\widetilde{f}_{k}(z)\right)&=\sum_{y\in [x]_{\cR}}f_{k}(x,y)\left(\sum_{w\in [y]_{\cR}}(f_{k}*f_{k})(y,w)\right)\\
&\leq \sum_{w,y\in [x]_{\cR}}f_{k}(x,y)f_{2k}(y,w)\\
&=\sum_{w\in [x]_{\cR}}(f_{k}*f_{2k})(x,w)\\
&\leq \sum_{w\in [x]_{\cR}}f_{3k}(x,w)\\
&=\widetilde{f}_{3k}(x),
\end{align*}
with all of the above inequalities and equalities holding for almost every $x\in X$.
So we have shown that
\[\widetilde{f}_{k}(x)^{2}\leq \widetilde{f}_{k}(x)^{-1}\psi_{k}(x)\widetilde{f}_{3k}(x) \mbox{ for almost every $x\in X$,}\]
and rearranging proves the desired inequality.

(\ref{item:mass transport 23}):
By the mass transport principle and $\pi$-symmetry:
\[\int \pi(x)\phi_{k}(x)\,d\mu(x)=\int \sum_{y\in [x]_{\cR}}\pi(x)f_{k}(x,y)\widetilde{f}_{k}(x)^{-1}\,d\mu(x)=\int \pi\,d\mu,\]
as
\[\sum_{y\in [x]_{\cR}}f_{k}(x,y)\widetilde{f}_{k}(x)^{-1}=1.\]
Similarly,
\begin{align*}&\int \pi(x)\psi_{k}(x)\,d\mu(x)\\=&\int \sum_{y\in [x]_{\cR}}\pi(x)f_{k}(x,y)\widetilde{f}_{k}(y)\left(\sum_{z\in [x]_{\cR}}f_{k}(x,z)\widetilde{f}_{k}(z)\right)^{-1}\,d\mu(x)=\int \pi\,d\mu,\end{align*}
as
\[\sum_{y\in [x]_{\cR}}f_{k}(x,y)\widetilde{f}_{k}(y)\left(\sum_{z\in [x]_{\cR}}f_{k}(x,z)\widetilde{f}_{k}(z)\right)^{-1}=1.\]
\end{proof}

We now prove Theorem \ref{thm:generalized 23}. It will be helpful to pass to limits along subsets of $\N$ which are ``not too sparse" in a multiplicative sense. We say that $A\subseteq\N$ is \emph{asymptotically dense on the logarithmic scale} if
\[\lim_{\substack{x\to \infty\\ x\in \R}}\inf_{n\in A}|x-\log(n)|=0.\]
We leave it as an exercise to the reader to show that if $A\subseteq \N$ is asymptotically dense on the logarithmic scale, and if  $(a_{k})_{k=1}^{\infty}$ is a sequence of nonnegative real numbers for which there is a uniform $C>0$ with
\[a_{l+k}\geq C^{l}a_{k}\mbox{ for all $l,k\in \N$},\]
then
\[\limsup_{k\in A}a_{k}^{1/k}=\limsup_{k\to\infty}a_{k}^{1/k},\mbox{ and } \liminf_{k\in A}a_{k}^{1/k}=\liminf_{k\to\infty}a_{k}^{1/k}.\]

\begin{proof}[Proof of Theorem \ref{thm:generalized 23}]
Adopt notation as in Lemma \ref{lem:weak submultiplicativity}. Set
\[\overline{f}(x)=\limsup_{k\to\infty}\widetilde{f}_{k}(x)^{1/2k},\mbox{ and } \underline{f}(x)=\liminf_{k\to\infty}\widetilde{f}_{k}(x)^{1/2k}.\]
For $p,q\in \N\cup\{0\},k\in \N$ we have, by Lemma \ref{lem:weak submultiplicativity} and induction, that
\begin{equation}\label{eqn:inductive 23 estimate}
\widetilde{f}_{k}(x)^{2^{p}3^{q}}\leq C_{p,q,k}(x)\widetilde{f}_{2^{p}3^{q}k}(x)\mbox{ for almost every $x\in X$,}
\end{equation}
where
\[C_{p,q,k}(x)=\prod_{i=0}^{p-1}\phi_{2^{i}k}(x)^{2^{p-1-i}3^{q}}\prod_{j=0}^{q-1}\psi_{2^{p}3^{j}k}(x)^{3^{q-1-j}}.\]
By Lemma \ref{lem:weak submultiplicativity} and the fact that $\pi\in L^{1}(X,\mu)$, both $\sum_{k}k^{-2}\pi\phi_{k},$ $\sum_{k}k^{-2}\pi\psi_{k}$ converge almost everywhere. Since $\pi(x)>0$ almost surely, we see that for almost every $x\in X$, there is a $k_{0}$ (depending upon $x$) so that for $k\geq k_{0}$ we have $\phi_{k}(x)\leq k^{2},$ $\psi_{k}(x)\leq k^{2}$. So for almost every $x\in X$, and for all $k\geq k_{0}$,  and all $p,q\in \N\cup\{0\}$:
\[C_{p,q,k}(x)\leq \prod_{i=0}^{p-1}(2^{i}k)^{2^{p-i}3^{q}}\prod_{j=0}^{q-1}(2^{p}3^{j}k)^{3^{q-1-j}2},\]
which implies
\begin{align*}\frac{1}{2^{p}3^{q}}\log C_{p,q,k}(x)&\leq \sum_{i=0}^{p-1}2^{-i}\log(2^{i}k)+\sum_{j=0}^{q-1}2^{1-p}3^{-1-j}\log(2^{p}3^{j}k)\\
&\leq B+2\log(k),
\end{align*}
where
\[B=\left(\max_{p\in \N\cup\{0\}}\frac{p}{2^{p}}\right)\log(2)+\log(2)\sum_{i=0}^{\infty}\frac{i}{2^{i}}+\frac{2\log(3)}{3}\sum_{j=0}^{\infty}\frac{j}{3^{j}}<\infty.\]
Fix $k_{1}\geq k_{0}$. Since $A=\{2^{p}3^{q}k_{1}:p,q\in \N\}$ is asymptotically dense on the logarithmic scale, we have by (\ref{eqn:inductive 23 estimate}) and hypothesis (\ref{item:small log scale errors}):
\begin{align}\label{eqn:23 comparison bw limsup and liminf}
  \underline{f}(x)=\liminf_{p+q\to \infty}\widetilde{f}_{2^{p}3^{q}k_{1}}(x)^{\frac{1}{2^{p}3^{q}k_{1}}}\geq \liminf_{p+q\to\infty}C_{p,q,k_{1}}(x)^{-\frac{1}{2^{p}3^{q}k_{1}}}\widetilde{f}_{k_{1}}(x)^{1/k_{1}}\\
\geq e^{-B/k_{1}}k_{1}^{-2/k_{1}}\widetilde{f}_{k_{1}}(x)^{1/k_{1}}.\nonumber
\end{align}

Letting $k_{1}\to \infty$, we see that $\underline{f}\geq \overline{f}$ almost everywhere, and this proves that $\widetilde{f}$ exists almost everywhere. Note that (\ref{eqn:23 comparison bw limsup and liminf}) applied with $k_{1}=k_{0}$ shows that $\underline{f}(x)\geq e^{-B/k_{0}}k_{0}^{-2/k_{0}}\widetilde{f}_{k_{0}}(x)^{1/k_{0}}.$ Thus $\widetilde{f}(x)>0$ for almost every $x$.

(\ref{item: average monotonicity}):
By (\ref{eqn:inductive 23 estimate}), we have for $p,k\in \N$ and almost every $x\in X$:
\begin{align*}
    \frac{\widetilde{f}_{k}(x)}{\widetilde{f}_{2^{p}k}(x)^{\frac{1}{2^{p}}}}\leq C_{p,0,k}(x)^{\frac{1}{2^{p}}}&=\left(\prod_{i=0}^{p-1}\phi_{2^{i+1}k}(x)^{2^{p-1-i}}\right)^{1/2^{p}}\\
    &\leq \frac{1}{2^{p}}\left(1+\sum_{i=0}^{p-1}2^{p-1-i}\phi_{2^{i+1}k}(x)\right),
    \end{align*}
where in the last step we use the arithmetic-geometric mean inequality. Lemma \ref{lem:weak submultiplicativity} implies
\[\int \pi(x)\frac{\widetilde{f}_{k}(x)}{\widetilde{f}_{2^{p}k}(x)^{\frac{1}{2^{p}}}}\,d\mu(x)\leq \int \pi\,d\mu,\]
letting $p\to\infty$ and applying Fatou's Lemma proves (\ref{item: average monotonicity}).

(\ref{item:integral formula for sup}):
Part (\ref{item: average monotonicity}) shows that
\[\int \pi(x) \widetilde{f}_{k}(x)\,d\mu(x)\leq \|\widetilde{f}\|_{\infty}^{k}\int \pi\,d\mu,\]
for every $k\in \N$. Since $\pi\in L^{1}(X,\mu),$
\[\limsup_{k\to\infty}\left(\int \pi(x)\widetilde{f}_{k}(x)\,d\mu(x)\right)^{\frac{1}{k}}\leq \|\widetilde{f}\|_{\infty}.\]
We now prove the reverse inequality. Fatou's Lemma implies that for all $r\in \N$:
\begin{align*}\int \pi(x)\widetilde{f}(x)^{r}\,d\mu(x)&\leq \liminf_{k\to \infty}\int \pi(x)\widetilde{f}_{k}(x)^{\frac{r}{k}}\,d\mu(x)\\&\leq \liminf_{k\to\infty}\left(\int \pi(x)\widetilde{f}_{k}(x)\,d\mu(x)\right)^{\frac{r}{k}}\left(\int \pi(x)\,d\mu(x)\right)^{1-\frac{r}{k}},\end{align*}
where in the last step we use Holder's inequality for $k>r$. Since $\pi\in L^{1}(X,\mu)$,
\[\left(\int \pi(x)\widetilde{f}(x)^{r}\,d\mu(x)\right)^{\frac{1}{r}}\leq \liminf_{k\to\infty}\left(\int \pi(x) \widetilde{f}_{k}(x)\,d\mu(x)\right)^{\frac{1}{k}}.\]
Letting $r\to\infty$ and using that $0<\pi(x)$ for a.e. $x$ and that $\pi\in L^{1}(X,\mu)$ shows that
\[\|\widetilde{f}\|_{\infty}\leq \liminf_{k\to\infty}\left(\int \pi(x)\widetilde{f}_{k}(x)\,d\mu(x)\right)^{\frac{1}{k}}.\]
\end{proof}

\subsection{Normal subrelations, the proof of Theorem \ref{T:existence of the relative spectral radius} (\ref{item: ergodic normalizer})}\label{sec:invariance of almost sure sr}

In this subsection, we prove that the co-spectral radius $\rho^{\cS}(x)$ is almost surely constant when $\cS$ is either ergodic normal (the case when $\cS$ is ergodic is fairly direct, see Corollary \ref{cor: local 23} (\ref{item: almost sure existence local cs})). We prove a common generalization of the cases where $\cS$ is normal or ergodic, for which we need partial one-sided normalizers. 

\begin{defn}\label{defn:1-sided normalizer}
Let $(X,\mu)$ be a standard probability space and $\cS\leq \cR$ discrete, measure-preserving equivalence relations on $(X,\mu)$. We define the \emph{partial one-sided normalizers} to be the set of $\phi\in [[\cR]]$ so that for almost every $x\in \dom(\phi)$ we have $\phi([x]_{\cS}\cap \dom(\phi))\subseteq [\phi(x)]_{\cS}$. We use $PN_{\cR}^{(1)}(\cS)$ for the set of partial one-sided normalizers of $\cS$ inside of $\cR$.
\end{defn}
One example to keep in mind for intuition is the following: suppose that $\cR$ is the orbit equivalence relation of a measure-preserving action of $G$ on $(X,\mu)$. Suppose that $H\leq G$, and let $\cS=\{(x,ax):a\in N,x\in X\}$. For $g\in G$, we let $\phi_{g}\in [\cR]$ be given by $\phi_{g}(x)=gx$. If $g$ is in the normalizer of $H$ inside $G$, then by direct calculation $\phi_{g}([x]_{\cS})=[x]_{\cS}$ for almost every $x\in X$. So $\phi_{g}\in PN_{\cR}^{(1)}(\cS)$. More generally, $\phi_{g}\big|_{E}\in PN_{\cR}^{(1)}(\cS)$ for every $g\in G$ in the normalizer of $H$. We can thus think of the partial one-sided normalizers as a generalization of normalizers to the setting of the full pseudogroup. 

Another source of elements in the partial normalizer is the following. Recall that the set of \emph{endomorphisms of $\cR$ over $\cS$}, denoted $\End_{\cR}(\cS)$ are the measurable functions $\phi\colon X\to X$ so that for almost every $x\in X$ the following two conditions hold:
\begin{itemize}
    \item $\phi(x)\in [x]_{\cR}$,
    \item $\phi([x]_{\cS})\subseteq [\phi(x)]_{\cS}$.
\end{itemize}

Such functions need not be injective modulo null sets, and when they fail to be injective modulo null sets are not measure-preserving. Indeed, one can use the mass-transport principle to show that $\phi\in \End_{\cR}(\cS)$ we have that 
\[\frac{d\phi_{*}\mu}{d\mu}(x)=|\phi^{-1}(\{x\})| \textnormal{ for almost every $x\in X$,}\]
(we will not use this fact so will not give a full proof of it). However, as we will see shortly (see Lemma \ref{lem:partial normalizer decomposition}) for each $\phi\in \End_{\cR}(\cS)$ we may find a partition modulo null sets $(E_{i})_{i\in I}$  of $X$ into countably many sets so that $\phi|_{E_{i}}$ is injective for each $i\in I$. 

This last example in fact provides the motivation for our consideration of the one-sided partial normalizers, as it is one way to define normal subrelations.

\begin{defn}[See Theorem 2.2 of {FSZ}]
Let $\cR$ be a discrete, measure-preserving equivalence relation on a standard probability space $(X,\mu)$. A subrelation $\cS\leq \cR$ is \emph{normal} if  there is a countable set $D\subseteq \End_{\cR}(\cS)$ so that $[x]_{\cR}=\bigcup_{\phi\in D}\phi([x]_{\cS})$ for almost every $x\in X$.
\end{defn}
From the above definition, one can imagine attempting to prove that the co-spectral is almost surely constant for a normal subrelation by trying to understand how it varies after applying endomorphisms of $\cS\leq \cR$. However, as alluded to above, endomorphisms can have undesirable properties (namely lack of injectivity and failure to be measure-preserving) and the possibility of these undesirability properties make them ill-suited for our proofs. A good tradeoff for our purposes is to drop being everywhere defined and gain being measure-preserving and injective. This is precisely the purpose of the one-sided partial normalizers.

We have a natural way for elements of $PN^{(1)}_{\cR}(\cS)$ to act on $L^{p}(X,\mu)$ for $1\leq p\leq \infty$. Namely, if $\phi \in PN^{(1)}_{\cR}(\cS)$ we define $\alpha_{\phi}\in B(L^{p}(X,\mu))$ via
\[(\alpha_{\phi}f)(x)=1_{\ran(\phi)}(x)f(\phi^{-1}(x)).\]
This allows us to say what it means for $PN^{(1)}_{\cR}(\cS)$ to act ergodically. It simply means that if $f\in L^{p}(X,\mu)$ and $\alpha_{\phi}(f)=1_{\ran(\phi)}f$ for every $\phi \in PN^{(1)}_{\cR}(\cS)$, then $f$ is essentially constant. We leave it as an exercise to the reader to follow the usual arguments to verify that this is equivalent to saying  there exists a countable $C\subseteq PN^{(1)}_{\cR}(\cS)$ so that if a measurable subset of $X$ is invariant under $C$, then its measure is either $0$ or $1$. Note that $PN^{(1)}_{\cR}(\cS)$ automatically acts ergodically if $\cS$ is ergodic, because $[\cS]\subseteq PN^{(1)}_{\cR}(\cS)$. The following lemma shows that $PN^{(1)}_{\cR}(\cS)$ also acts ergodically if $\cS$ is normal and $\cR$ is ergodic. 
\begin{lem}\label{lem:partial normalizer decomposition}
Suppose that $\cS$ is normal in $\cR$. Then there is a countable set $C\subseteq PN^{(1)}_{\cR}(\cS)$ so that
\[[x]_{\cR}=\bigcup_{\phi\in C}\phi(x)\]
for almost every $x\in X$.

\end{lem}

\begin{proof}
By \cite[Theorem 2.2]{FSZ}, we may find a countable $D\subseteq \End_{\cR}(\cS)$ so that
\[[x]_{\cR}=\bigcup_{\phi\in D}[\phi(x)]_{\cS}\]
for almost every $x\in X$.
Fix a countable subgroup $H\subseteq [\cS]$
 so that $[x]_{\cS}=Hx$ for almost every $x\in X$. Then for almost every $x\in X$ we have that
\[[x]_{\cR}=\bigcup_{\phi\in D,\psi\in H}(\psi\circ \phi)(x).\]
Fix a countable subgroup $G\subseteq [\cR]$ so that $[x]_{\cR}=Gx$ for almost every $x\in X$. Then for each $\phi \in D$ we may write, up to sets of measure zero,
\[\dom(\phi)=\bigcup_{g\in G}E_{g,\phi}\]
where $E_{g,\phi}\subseteq \{x\in \dom(\phi):\phi(x)\in gx\}$.  For each $g\in G$, $\phi\in D$, define $\phi_{g}\in [[\cR]]$ by declaring that $\dom(\phi_{g})=E_{g,\phi}$ and that $\phi_{g}(x)=gx$ for $x\in E_{g,\phi}$.
 Observe that $\phi_{g}\in PN^{(1)}_{\cR}(\cS)$. Then
\[[x]_{\cR}=\bigcup_{\substack{\psi\in H,\\ \phi\in D,g\in G \textnormal{ with } x\in E_{g,\phi}}}\psi\circ \phi_{g}(x).\]
Thus
\[C=\{\psi\circ \phi_{g}:g\in G,\phi\in D,\psi\in H\}\]
is the desired countable set.

\end{proof}

By the preceding lemma and the inclusion $[\cS]\subseteq PN^{(1)}_{\cR}(\cS)$ the condition that $PN^{(1)}_{\cR}(\cS)$ act ergodically on $(X,\mu)$ is satisfied if either $\cS$ is ergodic or $\cS$ is normal and $\cR$ is ergodic.
We will show that $\rho^{\cS}$ is almost surely invariant under the partial one-sided normalizers, and is thus essentially constant when $PN^{(1)}_{\cR}(\cS)$ acts ergodically.

Since the elements of $PN^{(1)}_{\cR}(\cS)$ are only partially defined, we will need to relate the co-spectral to the modification given by Example \ref{example:local cs} by restricting to a subset. It will not be hard to show from Theorem \ref{thm:generalized 23} that $\rho_{E}^{\cS}(x)$ as defined in Example \ref{example:local cs} exists for almost every $x\in X$. To deduce the existence of $\rho^{\cS}_{E}(x)$ for almost every $x\in X$, as well as some of its basic properties, we need to recall the \emph{$\cS$-saturation} of a measurable set.
Given a discrete, measure-preserving equivalence relation $\cS$ on a standard probability space and a measurable $E\subseteq X$, the \emph{$\cS$-saturation of $E$} is the (unique up to measure zero sets) measurable subset $\widetilde{E}\subseteq X$ satisfying the following  properties:
\begin{itemize}
    \item $E\subseteq \widetilde{E}$,
    \item for almost every $x\in \widetilde{E}$ we have $[x]_{\cS}\subseteq \widetilde{E}$,
    \item for almost every $x\in \widetilde{E}$, there is a $y\in E$ with $x\in [y]_{\cS}$.
\end{itemize}
If $\Gamma\leq [\cS]$ is a countable subgroup which generates $\cS$, then a model for $\widetilde{E}$ can be given by
\[\widetilde{E}=\bigcup_{\phi\in \Gamma}\phi(E).\]

It will also be helpful to relate $\rho^{\cS}_{E}$ to a quantity more operator-theoretic in nature.
For a measurable set $E\subseteq X$ of positive measure, and $x\in X$ we let
\[p_{2k,x,\cS,E}=\sum_{y\in [x]_{\cS}\cap E}p_{2k,x,y}.\]
We also let
\[\rho_{E}(\cR/\cS,\nu)=\lim_{k\to\infty}\left(\frac{1}{\mu(E)}\int_{E}p_{2k,x,\cS,E}\,d\mu(x)\right)^{1/2k}.\]
Define $\zeta_{E}\in L^{2}(\cR/\cS)$ by $(\zeta_{E})_{x}=\frac{1_{E}(x)}{\sqrt{\mu(E)}}\delta_{[x]_{\cS}}$. Then
\[\frac{1}{\mu(E)}\int_{E}p_{2k,x,\cS,E}\,d\mu(x)=\ip{\lambda_{\cS}(\nu)^{2k}\zeta_{E},\zeta_{E}},\]
so it follows from the spectral theorem that the limit defining $\rho_{E}(\cR/\cS,\nu)$ exists.
We now proceed to show that $\lim_{k\to\infty}p_{2k,x,\cS,E}^{\frac{1}{2k}}$ exists.
\begin{cor}\label{cor: local 23}
Let $(X,\mu)$ be a standard probability space, and let $\cS\leq \cR$ be discrete, probability measure-preserving equivalence relations defined over $X$. 
Let $\nu\in \Prob([\cR])$ be symmetric, countably supported, and generate $\cR$. 
Fix a measurable $E\subseteq X$ with positive measure.
\begin{enumerate}[(a)]
\item For almost every $x\in X$ we have that
\[\rho^{\cS}_{E}(x)=\lim_{k\to\infty}p_{2k,x,\cS,E}^{1/2k}\]
exists and is almost surely $\cS$-invariant.  \label{item: almost sure existence local cs}
\item  $\|\rho^{\cS}_{E}\|_{\infty}=\rho_{E}(\cR/\cS,\nu)$.\label{item:op norm as sup norm}
\end{enumerate}
\end{cor}

\begin{proof}
Note that the sequence of functions $f_{k}\colon E\to [0,\infty)$ given by
\[f_{k}(x,y)=p_{2k,x,y}1_{\cS}(x,y)1_{E}(x)1_{E}(y)\]
satisfies the hypothesis of Theorem \ref{thm:generalized 23} with $\cR$ replaced by $\cR|_{E}$.
This proves that $\rho^{\cS}_{E}(x)$ exists for almost every $x\in E$.
Since
\[\frac{1}{\mu(E)}\int p_{2k,x,\cS,E}\,d\mu(x)=\ip{\lambda_{\cS}(\nu)^{2k}\zeta_{E},\zeta_{E}},\]
Theorem \ref{thm:generalized 23} also proves (\ref{item:op norm as sup norm}).

 To prove that $\rho^{\cS}_{E}(x)$ exists for almost every $x$, let $\widetilde{E}$ be the $\cS$-saturation of $E$. Note that $\rho^{\cS}_{E}(x)=0$ for almost every $x\in \widetilde{E}^{c}$.
For almost every $x\in \widetilde{E}$, we have that $\rho^{\cS}_{E}(y)$ exists for all $y\in [x]_{\cS}\cap E$. Fix such an $x\in\widetilde{E}$, and let $y\in [x]_{\cS}\cap E.$ 
Then there is an $\ell\in \N$ with $p_{\ell,x,y}>0$. So for all $k\geq 2\ell$,
\[p_{2k,x,\cS,E}=\sum_{z\in [x]_{\cS}\cap E}p_{2k,x,z}\geq p_{\ell,x,y}^{2}\sum_{z\in [x]_{\cS}\cap E}p_{2(k-\ell),y,z}=p_{\ell,x,y}^{2}p_{2(k-\ell),y,\cS,E}. \]
This shows that for almost every $x\in \widetilde{E}$ we have
\[\liminf_{k\to \infty}p_{2k,x,\cS,E}^{\frac{1}{2k}}\geq \rho^{\cS}_{E}(y).\]
The proof that
\[\limsup_{k\to\infty}p_{2k,x,\cS,E}^{\frac{1}{2k}}\leq \rho^{\cS}_{E}(y)\]
is similar.
\end{proof}

It will be helpful to study how $\rho_{E}(\cR/\cS,\nu)$ varies as a function of $E$. We can embed measurable subsets (modulo null sets) of $(X,\mu)$ into $L^{1}(X,\mu)$ via identifying each set with its indicator function. We thus have a natural distance on measurable sets modulo null sets by
\[d(E,F)=\mu(E\Delta F).\]
We show that $\rho_{E}(\cR/\cS,\nu)$ is semicontinuous as function of $E$. For later use, it will also be useful to know that $\rho_{E}(\cR/\cS,\nu)$ as a function of $\cS$.
We now define a topology on subrelations of $\cR$ making this precise.

For $\cR$ a discrete, probability measure-preserving equivalence relation on $(X,\mu)$, $\phi\in [[\cR]]$, and $\cS\leq \cR$, set
\[F_{\cS,\phi}=\{x\in \dom(\phi):\phi(x)\in [x]_{\cS}\}.\]
For a given subrelation $\cS\leq \cR$, a finite set $\Omega\subseteq [[\cR]]$, and $\varepsilon>0$, let $\cO_{\Omega,\varepsilon}(\cS)$  consist of all subrelations $\cS'\leq \cR$ such that 
\[\mu(F_{\cS,\phi}\Delta F_{\cS',\phi})<\varepsilon \textnormal{ for all $\phi \in \Omega$}.\]
We define a topology on subrelations of $\cR$ (modulo null sets) by declaring that the family $\cO_{\Omega,\varepsilon}(\cS)$ form a neighborhood basis of $\cS$. Suppose that $I$ is countable, and that $\Phi=(\phi_{i})_{i\in I}$ in $[[\cR]]$ is such that $[x]_{\cR}=\{\phi_{i}(x):i\in I,x\in \dom(\phi_{i})\}$ for almost every $x\in X$. In this case, given $(\alpha_{i})_{i\in I}\in (0,1]^{I}$ with $\sum_{i}\alpha_{i}=1$, we can define a metric on subrelations of $\cR$ by 
\[d(\cS,\cS')=\sum_{i}\alpha_{i}\mu(F_{\cS,\phi_{i}}\Delta F_{\cS',\phi_{i}}).\]
It can be check that this metric induces the topology defined above (this metric is moreover complete, though we will not need this). 
We now prove our promised semicontinuity.

\begin{lem}\label{item: monotone continuity}
Let $(X,\mu)$ be a standard probability space and $\cR$ a discrete, measure-preserving equivalence relation define over $X$. Let $\nu\in \Prob([\cR])$ be symmetric, countably supported, and generate $\cR$. The map $(E,\cS)\mapsto \rho_{E}(\cR/\cS,\nu)$ is lower semicontinuous (where the domain is all pairs of positive measure, measurable subsets $E$ of $X$ and subrelations $\cS$ of $\cR$).
\end{lem}

\begin{proof}
From 
\[\frac{1}{\mu(E)}\int_{E} p_{2k,x,\cS,E}\,d\mu(x)=\ip{\lambda_{\cS}(\nu)^{2k}\zeta_{E},\zeta_{E}},\]
and the spectral theorem we have 
\[\rho_{E}(\cR/\cS,\nu)=\sup_{k}\left(\frac{1}{\mu(E)}\int_{E} p_{2k,x,\cS,E}\,d\mu(x)\right)^{1/2k}.\]
It thus suffices to prove that 
\[(E,
\cS)\mapsto \int_{E} p_{2k,x,\cS,E}\,d\mu(x)
\]
is continuous.
Let $\id_{E}\in [[\cR]]$ be defined by $\dom(\id_{E})=E$ and $\id_{E}(x)=x$ for every $x\in E$. Then $ \int_{E} p_{2k,x,\cS,E}\,d\mu(x)$ can be rewritten as 
\[\sum_{\phi\in \supp(\nu^{*2k})}\nu^{*2k}(\phi)\mu(E\cap F_{\cS,\id_{E}\phi}).\]
Note that each term in this sum is a continuous as a function of $(E,\cS)$. Since 
$\sum_{\phi\in \supp(\nu^{*2k})}\nu^{*2k}(\phi)=1,$
it follows that the sum converges uniformly. 
Thus
\[(E,\cS)\mapsto \int_{E} p_{2k,x,\cS,E}\,d\mu(x)
\]
is continuous. 

\end{proof}

For technical reasons, in order to show that $\rho^{\cS}$ is invariant under the partial normalizer of $\cS$ inside of $\cR$, it will be helpful to reduce to the case that $\nu$ is lazy.  We will briefly need to adopt notation for how the co-spectral radius depends upon the measure. So for $\cS\leq\cR,X$ as in Theorem \ref{thm: satuartion invariance} and a countably supported $\nu\in \Prob([\cR])$, we use $\rho^{\cS}_{\nu}(x)$ for the co-spectral radius at $x$ defined using $\nu$. Similarly, for $n\in \N$ we use $p_{n,x,\cS}^{\nu}$ for the probability that the random walk starting at $x$ associated to $\nu$ is in $[x]_{\cS}$ after $n$ steps.

It will be helpful to use the following lemma which shows that the local co-spectral radius associated to $\nu$ is uniformly close to the local co-spectral radius associated to $(1-t)\nu+t\delta_{\id}$ as $t\to 0$.

\begin{lem}\label{lem:laziness saves the day}
Let $(X,\mu)$ be a standard probability space, and $\cS\leq \cR$  discrete, probability measure-preserving equivalence relations defined over $X$. For a symmetric and countably supported $\nu\in \Prob([\cR])$ and $t\in [0,1]$, define $\nu_{t}=(1-t)\nu+t\delta_{\id}$. Then
\[\|\rho_{\nu_{t}}^{\cS}-\rho_{\nu}^{\cS}\|_{\infty}\leq t(1+\rho(\cR/\cS,\nu)),\mbox{ for all $t\in [0,1]$.}\]
\end{lem}

\begin{proof}
To simplify notation, set $\rho=\rho(\cR/\cS,\nu)$.
Let $\zeta_{E}$ be defined as in the discussion preceding Corollary \ref{cor: local 23}. Let $\eta\in \Prob([-\rho,\rho])$ be the spectral measure of $\lambda(\nu)$ with respect to $\zeta_{E}$.
Then for every measurable set $E\subseteq X$ of positive measure, we have that $\rho_{E}(\cR/\cS,\nu_{t}),\rho_{E}(\cR/\cS,\nu)$ are the $L^{\infty}$-norms of $s\mapsto s$, $s\mapsto (1-t)s+t$ with respect to $\eta$. In particular, for all measurable $E\subseteq X$ of positive measure
\[|\rho_{E}(\cR/\cS,\nu_{t})-\rho_{E}(\cR/\cS,\nu)|\leq t(1+\rho).\]
If $E$ is almost surely $\cS$-invariant, then $\rho^{\cS}_{E}=\rho^{\cS}1_{E}$. So Corollary \ref{cor: local 23} (\ref{item:op norm as sup norm}) implies that
\[|\|\rho^{\cS}_{\nu_{t}}1_{E}\|_{\infty}-\|\rho^{\cS}_{\nu}1_{E}\|_{\infty}|\leq t(1+\rho)\mbox{ for all measurable $E\subseteq X$ which are $\cS$-invariant}.\]

To prove the lemma, fix $c>0$ and let
\[E=\{x:\rho_{\nu}^{\cS}(x)+ct(1+\rho)\leq \rho^{\cS}_{\nu_{t}}(x)\}.\]
Note that $E$ is $\cS$-invariant, by $\cS$-invariance of $\rho^{\cS}_{\nu}$ and $\rho^{\cS}_{\nu_{t}}$.
Assume, for the sake of contradiction, that $\mu(E)>0$. Then by the preceding paragraph we have
\[\|\rho^{\cS}_{\nu_{t}}1_{E}\|_{\infty}\leq t(1+\rho)+\|\rho^{\cS}_{\nu}1_{E}\|_{\infty}.\]
On the other hand, the definition of $E$ forces
\[\|\rho^{\cS}_{\nu_{t}}1_{E}\|_{\infty}\geq ct(1+\rho)+\|\rho^{\cS}_{\nu}1_{E}\|_{\infty}.\]
Since $c>1$, we obtain a contradiction. Thus
\[\rho^{\cS}_{\nu_{t}}-\rho^{\cS}_{\nu}\leq ct(1+\rho)\]
almost everywhere. A similar proof shows that
$\rho^{\cS}_{\nu}-\rho^{\cS}_{\nu_{t}}\leq ct(1+\rho)$
almost everywhere.
Thus
\[\|\rho^{\cS}_{\nu_{t}}-\rho^{\cS}_{\nu}\|_{\infty}\leq ct(1+\rho)\]
for all $c>1$, and the proof is completed by letting $c\to 1$.

\end{proof}

The above reduction to the lazy case and the semicontinuity in Lemma \ref{item: monotone continuity} allow us to give a general formula for the local co-spectral radius in terms of the co-spectral radius.
We will ultimately use to prove invariance under partial normalizers of $\cS\leq \cR$ by restrict to sets where we have uniform lower bound on transition probabilities $p_{k,x,\phi(x)}$ for $\phi\in PN^{(1)}_{\cR}(\cS)$.

\begin{thm}\label{thm: satuartion invariance}
Let $(X,\mu)$ be a standard probability space, and let $\cS\leq \cR$ be discrete, probability measure-preserving equivalence relations defined over $X$, and fix a symmetric and countably supported $\nu\in \Prob([\cR])$.  
If $E\subseteq X$ has positive measure, then
\[\rho^{\cS}_{E}=\rho^{\cS}1_{\widetilde{E}},\]
where $\widetilde{E}$ is the $\cS$-saturation of $E$.

\end{thm}

\begin{proof}
By Lemma \ref{lem:laziness saves the day}, we may assume that $\nu$ is lazy. Observe that by $\cS$-invariance of $\widetilde{E}$, we have that $\rho^{\cS}1_{\widetilde{E}}=\rho^{\cS}_{\widetilde{E}}$ almost everywhere. Hence,
it is enough to show that $\rho^{\cS}_{E}=\rho^{\cS}_{\widetilde{E}}$ almost everywhere.
Clearly
$\rho^{\cS}_{E}\leq \rho^{\cS}_{\widetilde{E}}$
so it suffices to show that reverse inequality holds almost everywhere.
For $k\in \N$, let
\[E_{k}=\{x\in X:p_{2k,x,\cS,E}>0\}.\]
 Then $E_{k}$ is an increasing sequence of sets with $E_{k}\supseteq E$. By symmetry and laziness of $\nu$  we know that $\bigcup_{k}E_{k}=\widetilde{E}$ up to sets of measure zero.
 By Lemma \ref{item: monotone continuity}, it suffices to show that for almost every $k\in\N$ we have that  $\rho^{\cS}_{E_{k}}\leq \rho^{\cS}_{E}$ almost surely.

For $m\in \N$, set $E_{k,m}=\{x\in X:p_{2k,x,\cS,E}>1/m\}$. Then $E_{k,m}$ are increasing and
\[E_{k}=\bigcup_{m=1}^{\infty}E_{k,m},\]
so again by  Lemma \ref{item: monotone continuity}, it suffices to show that $\rho^{\cS}_{E_{k,m}}\leq \rho^{\cS}_{E}$ almost surely. Fix an $x\in E_{k,m}$ so that both $\rho^{\cS}_{E}(x),\rho^{\cS}_{E_{k,m}}(x)$ exist. Then for every $l\in \N$:
\[p_{2(l+k),x,\cS,E}\geq \sum_{y\in [x]_{\cS}\cap E_{k,m}}p_{2l,x,y}p_{2k,y,\cS,E}\geq \frac{1}{m}p_{2l,x,\cS,E_{k,m}}.\] 
Taking $2l^{th}$ roots of both sides and letting $l\to\infty$ shows that
\[\rho^{\cS}_{E}\geq \rho^{\cS}_{E_{k,m}}.\]
Since $\rho^{\cS}_{E},\rho^{\cS}_{E_{k,m}}$ exist almost everywhere, this completes the proof.

\end{proof}


We now have amassed enough results to show that the co-spectral radius does not increase after applying partial normalizing elements. 

\begin{prop}\label{prop:subadditive invariance}
Let $(X,\mu)$ be a standard probability space, and let $\cS\leq \cR$ be discrete, probability measure-preserving equivalence relations defined over $X$, and fix a symmetric countably supported $\nu\in \Prob([\cR])$ which generates $\cR$. Then for every $\phi \in PN^{(1)}_{\cR}(\cS)$ we have
\[\rho^{\cS}(\phi(x))\geq \rho^{\cS}(x)\]
for almost every $x\in \dom(\phi)$.
\end{prop}

\begin{proof}
First assume that $\nu$ is lazy.
Fix $\phi\in PN^{(1)}_{\cR}(\cS)$, replacing $X$ with a conull set we may assume that $\phi([x]_{\cS}\cap\dom(\phi))\subseteq [\phi(x)]_{\cS}$ for every $x\in X$. Since $\nu$ is lazy and generating, for almost every $x\in \dom(\phi)$ we have that $p_{k,x,\phi(x)}>0$ for all large $k$. So, up to sets of measure zero,
\[\dom(\phi)=\bigcup_{k=1}^{\infty}\{x\in \dom(\phi):p_{k,x,\phi(x)}>0\mbox{ for all $k\geq n$}\}.\]
Since $\nu$ is lazy, this union  is increasing and thus it follows that we may find a sequence $\delta_{n}\to 0$ of positive numbers and an increasing sequence of integers $r_{n}$ so that
\[\sum_{n=1}^{\infty}\mu(\{x\in \dom(\phi):p_{r_{n},x,\phi(x)}\leq \delta_{n}\})<\infty.\]
Set $F_{n}=\{x\in \dom(\phi):p_{r_{n},x,\phi(x)}>\delta_{n}\}$, and $E_{n}=\phi(F_{n}).$ By Borel-Cantelli, for almost every $x\in \dom(\phi)$ we have that $x\in F_{k}$ for all sufficiently large $k.$ By  Theorem \ref{thm: satuartion invariance}, for almost every $x\in \dom(\phi)$ the follow conditions are satisfied for all sufficiently large $k$:
\begin{itemize}
\item $x\in F_{k}$
\item $\rho^{\cS}(\phi(x))=\rho^{\cS}_{E_{k}}(\phi(x))$,
\item $\rho^{\cS}(x)=\rho^{\cS}_{F_{k}}(x)$.
\end{itemize}
Fix an $x\in \dom(\phi)$ which satisfies the above three conditions, and fix $k$ such that the above three bulleted items hold. Then for all $l\in \N$,
\[p_{2l,\phi(x),\cS,E_{k}}=\sum_{z\in [\phi(x)]_{S}\cap E_{k}}p_{2l,\phi(x),z}.\]
Since $\phi([x]_{\cS}\cap\dom(\phi))\subseteq [\phi(x)]_{\cS}$, and $\phi(x)\in E_{k},x\in F_{k}$, we have for all $l\geq r_{k}$:
\[p_{2l,\phi(x),\cS,E_{k}}\geq \sum_{y\in [x]_{S}\cap F_{k}}p_{2l,\phi(x),\phi(y)}\geq \delta_{k}^{2}\sum_{y\in [x]_{S}\cap F_{k}}p_{2(l-r_{k}),x,y}=\delta_{k}^{2}p_{2(l-r_{k}),x,\cS,F_{k}}.\]
where in the second to last step we use symmetry and the fact that $x\in F_{k}$. Thus:
\[\rho^{\cS}(\phi(x))=\rho^{\cS}_{E_{k}}(\phi(x))\geq \rho^{\cS}_{F_{k}}(x)=\rho^{\cS}(x).\]
So
\[\rho^{\cS}\circ \phi\big|_{\dom(\phi)}\geq \rho^{\cS}\big|_{\dom(\phi)}\]
almost everywhere. The general case follows from the lazy case by using Lemma \ref{lem:laziness saves the day}.
\end{proof}

We are now ready to prove that the co-spectral radius does not change under the partial one-sided normalizers. 

\begin{cor}\label{cor:normalizer invariance}
Let $(X,\mu)$ be a standard probability space, and let $\cS\leq \cR$ be discrete, probability measure-preserving equivalence relations defined over $X$, and fix a symmetric and countably supported $\nu\in \Prob([\cR])$ which generates $\cR$.
Suppose that $PN^{(1)}_{\cR}(\cS)$ acts ergodically on $(X,\mu)$. Then $\rho^{\cS}$ is almost surely constant (in particular, by Lemma \ref{lem:partial normalizer decomposition} this applies if $\cS$ is normal in $\cR$ and $\cR$ is ergodic).

\end{cor}

\begin{proof}
Let $C$ be as in Lemma \ref{lem:partial normalizer decomposition}.
By the preceding proposition and countability, for every $t\in [0,1]$
\[E_{t}=\{x\in X:\rho^{\cS}(x)\geq t\}\]
is invariant under $C$, and so by ergodicity of $\cR$ has measure $0$ or $1$ for every $t\in [0,1]$. If $s=\sup\{t:\mu(E_{t})=1\}$, then $\rho^{\cS}\geq s$ almost everywhere, and $\rho^{\cS}(x)<s+\frac{1}{n}$ for almost every $x$ and every $n\in \N$. Thus $\rho^{\cS}(x)=s$ for almost every $x$.
\end{proof}

We close this section with an example illustrating the fact that $\rho^{\cS}(x)$ may fail to be essentially constant if we only assume that $\cR$ is ergodic. Thus we need to assume something special about the inclusion $\cS\leq \cR$.

\begin{example}\label{example: direct sum counterexample}
Let $\cR$ be an ergodic, discrete, measure-preserving equivalence relation on $(X,\mu)$. Let $E\subseteq X$ be a measurable set of positive measure. Let
\[\cS=\cR\cap (E\times E)\cup\{(x,x):x\in E^{c}\}.\]
We claim the following holds.
\emph{Claim:}
\begin{enumerate}[(i)]
    \item \emph{for almost every $x\in E^{c}$, we have $\rho^{\cS}(x)=\rho(\cR,\nu)$,} \label{item: counterexample return to self}
    \item \emph{for almost every $x\in E$, we have $\rho^{\cS}(x)=1.$}\label{item: counterexample return to set}
\end{enumerate}
In particular, if $\cR$ is not hyperfinite, then by \cite[Lemma 2.2.]{BHAExtension} there is a $\nu\in \Prob([\cR])$ so that $\rho(\cR,\nu)<1$ (this also follows from condition (GM) in \cite{KaimanovichAmenability} being equivalent to hyperfiniteness) and this gives an example where the co-spectral radius is not essentially constant.
\begin{proof}[Proof of claim]
Let $\mathcal{T}=\{(x,x):x\in X\}$. Let $\lambda_{x}(\nu)$ be the Markov operator associated to $\nu$ acting on $\ell^{2}([x]_{\cR})$. Then for all $x\in X$ we have
\[\ip{\lambda_{x}(\nu)\delta_{x},\delta_{x}}=p_{2n,x,\cT},\]
and so the co-spectral radius of $\cR$ with respect to $\cT$ agrees with the spectral radius of $[x]_{\cR}$ with respect to $\nu$, i.e. $\rho^{\cT}(x)=\|\lambda_{x}(\nu)\|$. If $(x,y)\in \cR$, then $\lambda_{x}(\nu),\lambda_{y}(\nu)$ are both operators on $\ell^{2}([x]_{\cR})$ and $\lambda_{x}(\nu)=\lambda_{y}(\nu)$. So, by ergodicity, the operator norm $\|\lambda_{x}(\nu)\|$ is essentially constant. Since $\rho(\cR/\cT,\nu)=\rho(\cR,\nu)$, it follows from Theorem \ref{T:existence of the relative spectral radius} (\ref{item:almost sure cs radius})
that $\|\lambda_{x}(\nu)\|$ is almost surely equal to $\rho(\cR,\nu)$.

(\ref{item: counterexample return to self}):
For almost every $x\in E^{c}$ we have $\rho^{\cS}(x)=\rho^{\cT}(x)$, so this follows by the preceding paragraph.

(\ref{item: counterexample return to set}):
For almost every $x\in E$ we have $\rho^{\cS}(x)=\rho_{E}^{\cR}(x)$, with notation as in Corollary \ref{cor: local 23}. So this follows from Theorem \ref{thm: satuartion invariance}.
\end{proof}
\end{example}

\section{Co-spectral radius and Hyperfinite subrelations}\label{sec: hyperfinite relations}

In this section, we investigate the relation between co-spectral radius and hyperfiniteness. Specifically, we show that if $\cS$ is hyperfinite, then $\rho(\cR/\cS,\nu)=\rho(\cR,\nu)$.





\begin{prop}\label{prop:hyperfintie doesn't change sr}
Let $\cR$ be an ergodic, discrete measure-preserving equivalence relation over a standard probability space $(X,\mu).$ Suppose that $\nu\in \Prob([\cR])$ is countably supported, symmetric,  and that $\mathcal{S}\leq\cR$ is hyperfinite. Then $\rho(\cR,\nu)=\rho(\cR/\mathcal{S},\nu).$

\end{prop}

\begin{proof}
The inequality $\rho(\cR,\nu)\leq \rho(\cR/\mathcal{S},\nu)$ is clear, so it remains to prove the reverse inequality.
Since $\mathcal{S}$ is hyperfinite, we can write $\mathcal{S}=\bigcup_{n}\mathcal{S}_{n}$ where $\mathcal{S}_{n}\leq \cR$ is an increasing sequence and $\cS_{n}$ is an equivalence relation where almost every equivalence class is finite. 
Note that $\cS_{n}$ converges to $\cS$ in the topology defined by Lemma \ref{item: monotone continuity}. Thus by Lemma \ref{item: monotone continuity}, it is enough to show that $\rho(\cR/\mathcal{S}_{n},\nu)\leq \rho(\cR,\nu)$ for every $n\in \N.$ For an integer $\ell$, let $E_{\ell}=\{x\in X:|[x]_{\cS_{n}}|\leq \ell\}$. Then $E_{k}$ are an increasing sequence of measurable subsets whose union is conull. By Theorem \ref{T:existence of the relative spectral radius}, we know that $\rho(\cR/\cS_{n},\nu)=\|\rho^{\cS_{n}}\|_{\infty}$. So it suffices to show that $\|\rho^{\cS_{n}}1_{E_{\ell}}\|_{\infty}\leq \rho(\cR,\nu)$ for all $\ell,n$.

Fix integer $n,\ell\in \N.$ We may then choose $\phi_{1},\cdots,\phi_{\ell}\in [[\cS_{n}]]$ so that for every $x\in E_{\ell}$,
\[[x]_{\cS_{n}}=\{\phi_{j}(x):1\leq j\leq \ell,x\in \dom(\phi_{j})\},\]
and with $\phi_{1}=\id$. 
Then, for every $x\in E_{\ell}$, and every $k\in \N$,
\[p_{2k,x,\cS_{n}}=\sum_{j=1}^{\ell}p_{2k,x,\phi_{j}(x)}1_{\dom(\phi_{j})}(x)\]
Set 
\[\rho^{\cS_{n}}(x)=\lim_{n\to\infty}p_{2k,x,\cS_{n}}^{1/2k},\textnormal{ and } \rho(x)=\lim_{n\to\infty}p_{2k,x,x}^{1/2k}.\]
Then:
\[\rho^{\cS_{n}}(x)=\lim_{k\to\infty}p_{2k,x,\cS}^{1/2k}=\lim_{k\to\infty}\left(\sum_{j=1}^{\ell}p_{2k,x,\phi_{j}(x)}1_{\dom(\phi_{j})})(x)\right)^{1/2k}=\lim_{k\to\infty}\max_{j}\left(p_{2k,x,\phi_{j}(x)}1_{\dom(\phi_{j})}\right)^{1/2k}.\]
For each $j\in \{0,\cdots,\ell\}$ such that $x\in \dom(\phi_{j})$, we may choose a non-negative integer $t_{j}$ so that 
$p_{t_{j},x,\phi_{j}(x)}>0$. Thus
\[p_{2k,x,\phi_{j}(x)}\leq p_{2(k+t_{j}),x,x}p_{t_{j},x,\phi_{j}(x)}^{-2}.\]
Thus
 \[\rho^{\cS_{n}}_{E_{\ell}}(x)\leq \rho(x)\]
for almost every $x$. Hence
\[\|\rho^{\cS_{n}}1_{E_{\ell}}\|_{\infty}\leq \|\rho\|_{\infty}=\rho(\cR,\nu).\]

\end{proof}

In the group context, the analogous statement is that if $\Gamma$ is a countable, discrete group and $\nu \in \Prob(\Gamma)$ is symmetric with $\ip{\supp(\nu)}=\Gamma$, then for any amenable $H\leq \Gamma$ we have that $\rho(\Gamma,\nu)=\rho(\Gamma/H,\nu)$. It is known that the converse fails, namely there are cases of $\Gamma,\nu$ and nonamenable $H\leq \Gamma$ so that $\rho(\Gamma,\nu)=\rho(\Gamma/H,\nu)$. It is a theorem of Kesten \cite[Theorem 2]{KestenSR} that the converse is true if we assume in addition that $H$ is normal. This was generalized to invariant random subgroups by Ab\'{e}rt-Glasner-Vir\'{a}g (see \cite{AVY14}). Normal subgroups and IRS's can both be realized as a special case of normal subrelations. So it is natural to ask if $\cR$ is an ergodic, probability measure-preserving, discrete relation and if $\nu\in \Prob([\cR])$ is symmetric and generating, and if $\cS\triangleleft \cR$ has $\rho(\cR,\nu)=\rho(\cR/\cS,\nu)$, do we necessarily have that $\cS$ is hyperfinite?  We will show that the answer is ``no" in general, but it might be helpful first to study a special case.

In general, a normal subrelation $\cS\triangleleft \cR$ can be expressed in terms of the partial one-sided normalizers of $\cS$ generating $\cR$ (recall our discussion for Section \ref{sec:invariance of almost sure sr}). Let us begin by investigating the case that $\cR$ is generated by $N_{\cR}(\cS)=\{\phi\in [\cR]:\phi([x]_{\cS})=[\phi(x)]_{\cS}\mbox{ for almost every $x\in X$}\}$. Moreover, to simplify things, assume that $\nu$ is the uniform measure on a finite, generating subset $\Phi$ of $N_{\cR}(\cS)$. It turns out that we can naturally translate this special case into a generalization of the Ab\'{e}rt-Glasner-Virag result.

\begin{prop}\label{prop:IRS from normality}
Let $\cR$ be a discrete, ergodic, probability measure-preserving equivalence relation over a standard probability space $(X,\mu)$. Let $\cS\leq \cR$, and let $\Gamma\leq N_{\cR}(\cS)$ be a countable subgroup. For $x\in X$, set $H_{x}=\{\phi\in\Gamma:\phi(x)\in [x]_{\cS}\}$.
\begin{enumerate}[(i)]
\item \label{item:field of IRS's from normality} For almost every $x\in X$, we have that $H_{x}$ is a subgroup of $\Gamma$. Moreover, for almost every $x\in X$ and every $\phi\in\Gamma$ we have $H_{\phi(x)}=\phi H_{x}\phi ^{-1}$.
\item \label{item:IRS's recover subreln} If $\Gamma$ generates $\cR$, then for almost every $x\in X$ we have that $[x]_{\cS}=H_{x}x$.
\item \label{item: translation of Kesten problem} Suppose that $\Gamma$ generates $\cR$, and that $\nu\in \Prob(\Gamma)$ is symmetric with $\ip{\supp(\nu)}=\Gamma$. Then for almost every $x\in X$ we have
\[\rho(\cR,\nu)=\rho(\Gamma/\Stab_{\Gamma}(x),\nu),\,\,\, \rho(\cR/\cS,\nu)=\rho(\Gamma/H_{x},\nu).\]
\end{enumerate}
\end{prop}

\begin{proof}
Items (\ref{item:field of IRS's from normality}),(\ref{item:IRS's recover subreln}) follow from \cite[Proposition 8.10]{RobinThesis}.


(\ref{item: translation of Kesten problem}): Since $\Gamma$ generates $\cR$, for almost every $x\in X$ we have a $\Gamma$-equivariant bijection $f\colon \Gamma/\Stab_{\Gamma}(x)\to [x]_{\cR}$ satisfying $f(\phi(x)\Stab_{\Gamma}(x))$ for every $\phi\in\Gamma$. Moreover, since $\Gamma\leq N_{\cR}(\cS)$, for almost every $x\in X$ we have for all $\phi,\psi\in\Gamma$ that $(\phi(x),\psi(x))\in\cS$ if and only if $(\psi^{-1}\phi(x),x)\in \cS$, which is equivalent to saying that $\phi H_{x}=\psi H_{x}$. Thus  $f$ induces a $\Gamma$-equivariant bijection $\overline{f}\colon \Gamma/H_{x}\to [x]_{\cR}/\cS$ satisfying $\overline{f}(\phi H_{x})=[\phi(x)]_{\cS}$.

For $H\leq \Gamma$, let $p_{k,H}$ be the probability that the random walk on $\Gamma$ corresponding to $\nu$ and starting at $1$ is at $H$ after $k$ steps. Then by the above paragraph we have for almost every $x\in X$ that
\[p_{k,x,x}=p_{k,\Stab_{\Gamma}(x)},\mbox{ and }\, p_{k,x,\cS}=p_{k,H_{x}}.\]
Thus for almost every $x\in X$ we have
\[\lim_{k\to\infty}p_{2k,x,x}^{1/2k}=\rho(\Gamma/\Stab(x)),\mbox{ and }\lim_{k\to\infty}p_{2k,x,\cS}^{\frac{1}{2k}}=\rho(\Gamma/H_{x}).\]
Since we are assuming that $\cR$ is ergodic and that $\Gamma\leq N_{\cR}(\cS)$ generates $\cR$, the result now follows from Theorem \ref{T:existence of the relative spectral radius} (\ref{item: ergodic normalizer}).
\end{proof}

The above Proposition motivates the following definition, which first appeared in \cite[Section 8]{RobinThesis} (under slightly different terminology). Recall that if $\Gamma$ is a countable, discrete group $\Sub(\Gamma)$ denotes the space of subgroups of $\Gamma$. We may identify each subgroup with its indicator function and thus view $\Sub(\Gamma)\subseteq \{0,1\}^{\Gamma}$. We may thus regard $\Sub(\Gamma)$ as a compact, metrizable space by giving $\Sub(\Gamma)$ the restriction of the product topology. An \emph{invariant random subgroup} is a Borel probability measure $\eta$ on $\Sub(\Gamma)$ which is invariant under the conjugation action of $\Gamma$ on $\Sub(\Gamma)$. We let $\IRS(\Gamma)$ be the space of invariant random subgroups of $\Gamma$.

\begin{defn}[Section 8 of \cite{RobinThesis}]
Suppose that $\eta_{1},\eta_{2}\in \IRS(\Gamma)$. A \emph{monotone joining of $\eta_{1}$ with $\eta_{2}$ is a $\zeta\in \Prob(\Sub(\Gamma)\times \Sub(\Gamma))$} which is invariant under $\Gamma\actson \Sub(\Gamma)\times \Sub(\Gamma)$ given by $g\cdot (K,H)=(gKg^{-1},gHg^{-1})$ and which satisfies
\[\zeta(\{(K,H):K\subseteq H\})=1.\]
\end{defn}
We will often use probabilistic language and think of the $\Sub(\Gamma)$-valued random variables $K,H$ with distribution $\eta_{1},\eta_{2}$ as the coupled IRS's. Thus we will often say ``let $K\leq H$ be a monotone joining of IRS's $H,K$".

In our context, given $\cS\leq \cR$ and $\Gamma\leq N_{\cR}(\cS)$ we get a monotone joining of IRS's by considering the pushforward of $\mu$ under the map $\Gamma\to \Sub(\Gamma)\times \Sub(\Gamma)$ given by $x\mapsto (\Stab_{\Gamma}(x),H_{x})$. We can also reverse this construction.

\begin{thm}[Theorem 8.15 of \cite{RobinThesis}]\label{prop:IRS realization}
Let $\zeta\in \Prob(\Sub(\Gamma)\times \Sub(\Gamma))$ be a monotone joining of IRS's $\eta_{1},\eta_{2}$. Then there is a standard probability space $(X,\mu)$, a probability measure-preserving action $\Gamma\actson (X,\mu)$ and a normal subrelation $\cS$ of the orbit equivalence relation of $\Gamma\actson (X,\mu)$ with the following property. We have that $\Gamma\leq N_{\cR}(\cS)$ and if we set
\[H_{x}=\{g\in \Gamma:(gx,x)\in \cS\},\]
 and define $\Theta\colon X\to \Sub(\Gamma)\times \Sub(\Gamma)$ by $\Theta(x)=(\Stab_{\Gamma}(x),H_{x})$, then $\zeta=(\Theta)_{*}(\mu)$.
\end{thm}

So by Proposition \ref{prop:IRS from normality} our question on co-spectral radii translates to the following. Suppose that $K\leq H$ is a monotone joining of IRS's, under what conditions do we have that $\rho(\Gamma/H,\nu)=\rho(\Gamma/K,\nu)$ almost surely? Let us first that this is always true when $K$ is coamenable in $H$.

\begin{prop}\label{prop:coamenability cs change}
Suppose that $K\leq H\leq \Gamma$ are countable, discrete groups and that $\nu\in \Prob(\Gamma)$. If $K$ is coamenable in $H,$ then $\rho(\Gamma/H,\nu)=\rho(\Gamma/K,\nu)$.
\end{prop}

\begin{proof}
To say that $K$ is coamenable in $H$ means that the trivial representation of $H$ is weakly contained in the quasi-regular representation of $H$ on $\ell^{2}(H/K)$. By induction of representations, it follows that the quasi-regular representation of $\Gamma$ on $\ell^{2}(\Gamma/H)$ is weakly contained in the quasi-regular representation of $\Gamma$ on $\ell^{2}(\Gamma/K)$ \cite[Example E.1.8 (ii),Theorem F.3.5]{BHV}. This implies \cite[Theorem F.4.4]{BHV} that
\[\rho(\Gamma/H,\nu)=\|\lambda_{H}(\nu)\|\leq \|\lambda_{K}(\nu)\|=\rho(\Gamma/K,\nu),\]
where $\lambda_{H}\colon \Gamma\to \mathcal{U}(\ell^{2}(\Gamma/H))$, $\lambda_{K}\colon \Gamma\to \mathcal{U}(\ell^{2}(\Gamma/H))$ are the quasi-regular representations. The reverse inequality is direct to argue, so this completes the proof.
\end{proof}

Notice that if $\Gamma,H_{x},\Stab(x)$ are as in  Theorem \ref{prop:IRS realization}, then having $\Stab(x)$ be coamenable in $H_{x}$ does not guarantee that $\cS$ is hyperfinite. So from Theorem \ref{prop:IRS realization} and Proposition \ref{prop:coamenability cs change}, we get an example of an equivalent relation $\cR$, a $\nu\in \Prob([\cR])$ and a normal $\cS\triangleleft \cR$, so that $\cS$ is not hyperfinite, and yet $\rho(\cR,\nu)=\rho(\cR/\cS,\nu)$. So a naive generalization of Kesten's theorem does not hold in this context. In fact, there is even a monotone joining of IRS's $K\leq H$ of $\Gamma$ so that $K\leq H$ is almost surely not co-amenable, and yet still $\rho(\Gamma/K,\nu)=\rho(\Gamma/H,\nu)<1$ for every finitely supported $\nu\in \Prob(\Gamma)$ whose support generates $\Gamma$.

\begin{example}\textbf{ Counterexample to the converse of Proposition \ref{prop:coamenability cs change}.} Let $G$ be a non-amenable finitely generated group and let $\Gamma=\F_{2}\wr G$ be the wreath product of $\F_{2}$ with $G$. Let $N:=\bigoplus_{\F_2}G$. Recall that $\Gamma=\F_2\rtimes N$ with $\gamma\in \F_2$ acting on $N$ by $\gamma((g_\lambda)_\lambda)\gamma^{-1}=(g_\lambda)_{\gamma\lambda}$. For any subset $\Sigma\subset \F_2$ let $N_\Sigma$ be the subgroup $\bigoplus_{\lambda\in\Sigma}G.$ By construction, $N_\Sigma$ is a normal subgroup of $N$ and for every $\gamma\in \F_2$ we have $\gamma N_\Sigma\gamma^{-1}=N_{\gamma\Sigma}$. This defines a $\Gamma$-equivariant map
\[N_{\bullet}\colon \{0,1\}^{\F_2}\to \Sub(\Gamma).\]
A percolation on $\F_2$ is random subset of $\F_2$ with distribution invariant under the left translations. 
For any percolation $P\in \{0,1\}^{\F_2}$ the group $N_P$ is an invariant random subgroup of $\Gamma$. Let $p<q\in (0,1)$ and let $(P,Q)\in \{0,1\}^{\F_2}\times \{0,1\}^{\F_2}$ be a $\Gamma$-invariant coupling of two Bernoulli percolations of parameters $p$ and $q$ respectively such that $P\subset Q$ a.s. This can be arranged by first choosing $P$ as Bernoulli percolation of parameter $p$ and then declaring $Q$ to be the union of $P$ and an independent copy of a Bernoulli percolation with parameter $\frac{q-p}{1-p}$. In this way we construct an invariant random couple of subgroups $N_{P}\subset N_{Q}$.

The set $Q\setminus P$ is infinite a.s. so the quotient $N_Q/N_P$ is the direct sum of infinitely many copies of $G$. In particular $N_P$ is almost surely not co-amenable in $N_Q$. Let $S$ be a finite generating set for $\Gamma$. We claim that
\[ \rho(\Gamma/N_P,S)=\rho(\Gamma/N_Q,S)=\rho(\Gamma/N,S).\]
This will follow quite quickly from the semi-continuity properties of the co-spectral radius. Since $P$ is a Bernoulli percolation, for any $n\in\mathbb N$ we will almost surely find $\gamma_n\in\F_2$ such that $\gamma P$ contains the $R$-ball around the identity. The sequence of subgroups $N_{\gamma_nP}$ converges to $N$ in $\Sub(\Gamma)$. On the other hand $\rho(\Gamma/N_{\gamma_nP},S)=\rho(\Gamma/N_{P},S)$, because $N_{\gamma_nP}=\gamma_n N_P\gamma_n^{-1}.$ By Lemma \ref{lem: semicont} we conclude that $\rho(\Gamma/N_P,S)\geq \rho(\Gamma/N,S)$. The reverse inequality is clear so $\rho(\Gamma/N_P,S)=\rho(\Gamma/N,S).$ In the same way we show that $\rho(\Gamma/N_Q,S)=\rho(\Gamma/N,S).$
\end{example}
\begin{lem}\label{lem: semicont}
Let $\Gamma$ be a countable group generated by a finite symmetric set $S$. Let $\Lambda_n$ be a sequence of subgroups of $\Gamma$ converging to a subgroup $\Lambda_\infty\subset \Gamma$ in $\Sub(\Gamma)$. We have
\[\liminf_{n\to\infty} \rho(\Gamma/\Lambda_n,S)\geq \rho(\Gamma/\Lambda_\infty,S).\]
\end{lem}
\begin{proof}
For any subgroup $\Lambda\subset \Gamma$ let $B_\Lambda(R)$ denote the $R$-ball around the trivial coset in the Schreier graph ${\rm Sch}(\Gamma/\Lambda, S).$ We will write $P$ for the Markov transition operator $P=\frac{1}{|S|}\sum_{s\in S} s.$
Let $\varepsilon>0$ and let $f\in \ell^2(\Gamma/\Lambda_\infty)$ be a non-zero finitely supported function such that \[\frac{\langle Pf, f\rangle}{\langle f, f\rangle}\geq \rho(\Gamma/\Lambda_\infty,S)-\varepsilon.\]
Choose $R>0$ such that the support of $f$ is contained in $B_{\Lambda_\infty}(R)$. For big enough $n$ we will have an isomorphism between labeled graphs
$\iota_n\colon B_{\Lambda_n}(R+1)\simeq B_{\Lambda_\infty}(R+1).$ Let $f_n=f\circ \iota_n$ be the pullback of $f$ to $\ell^2(\Gamma/\Lambda_n)$. Then,
\[\frac{\langle Pf_n, f_n\rangle}{\langle f_n, f_n\rangle}=\frac{\langle Pf, f\rangle}{\langle f, f\rangle}\geq \rho(\Gamma/\Lambda_\infty,S)-\varepsilon.\] We deduce that $\liminf_{n\to\infty}\rho(\Gamma/\Lambda_n,S)\geq \rho(\Gamma/\Lambda_\infty,S)-\varepsilon$. We finish the proof by taking $\varepsilon\to 0$.
\end{proof}

\section{Co-spectral radius and percolation}\label{sec:percolation}
Let $(G,o)$ be a transitive graph with the root at $o$ (resp. unimodular random graph). A invariant percolation $P$ is a random subset of edges of $G$, such that the distribution of $P$ is invariant under graph automorphisms (resp. invariant under the re-rooting equivalence relation).
\begin{thm}\label{thm:PercolationExp}
Let $P$ be an invariant bond percolation on a unimodular random graph $(\mathcal G,o)$ of degree at most $d$. Let $C$ be the connected component of $o$ in the percolation $P$. Let $X_n$ be the standard random walk on $\mathcal G$ starting at $o$. The limit
\[\rho_P:=\lim_{n\to\infty}\mathbb P(X_n\in C)^{1/{2n}} \]
exists almost surely.
\end{thm}
\begin{proof}
The fist step is to use Proposition \ref{prop:PercGraph} to construct an associated p.m.p equivalence relation with a graphing and a sub-relation that will allow us to apply Theorem \ref{thm:SubExponent}.
We now borrow the notation from the section \ref{sec:co spectral radius setup} and Proposition \ref{prop:PercGraph}. Let $(G,o),P$ be an instance of the percolation and let $x=((G,o),P,\lambda)\in \Omega_\#$ with an i.i.d. random coloring $\lambda$. The graph $(\mathcal G_x,x)$ is isomorphic to $(G,o)$ almost surely. Since $C=[x]_\cS$ almost surely, the probability of returning to the connected component of $P$ containing the root $o$ at time $n$ is the same as $p_{n,x,\cS}^{\nu}$. We have
\[\lim_{n\to\infty}\mathbb P(X_n\in C)^{1/{2n}} =\lim_{n\to\infty} \left(p_{2n,x,\cS}^\nu\right)^{1/2n}.\]
The limit on the right hand side exists for $\mu_\#$ almost all $x$, by virtue of Theorem \ref{thm:SubExponent}.
\end{proof}
\subsection{Critical values from spectral radius}

Using Theorem \ref{thm:PercolationExp}, we can define two new critical values of the Bernoulli bond percolation that fit nicely with the existing classical exponents like $p_c,p_u,p_{[2\to 2]}$ and $p_{\rm exp}$. We recall their definitions below.
The $\cB_p$ denotes the Bernoulli bond percolation with parameter $p\in[0,1]$. For simplicity we restrict to transitive rooted  graph $(G,o)$, but the definitions below can be easily adapted to all ergodic unimodular random graphs.
\begin{itemize}
    \item $p_u$ is defined as the infimum of $p\in[0,1]$ such that $\cB_p$ has a unique infinite connected component a.s.
    \item $p_c$ is defined as the infimum of $p\in[0,1]$ such that $\cB_p$ has an infinite connected component a.s.
    \item For $x,y\in V(G)$ let $\tau_p(x,y)$ denote the probability that $x$ and $y$ are connected in $\cB_p$. The exponent $p_{\rm exp}$ is the supremum of all $p$ such that $\tau_p(x,y)$ decays exponentially in $\dist(x,y)$.
    \item Consider the operator $T_p$ acting on compactly functions on $V(G)$ defined by \[T_p(\phi)(x)=\sum_{y\in V(G)}\tau_p(x,y)\phi(y).\] The exponent $p_{[2\to 2]}$ is the supremum of $p$ such that $T_p$ defines a bounded operator $\ell^2(V(G))\to \ell^2(V(G))$ \cite[Section 2]{Hutchcroft}.
\end{itemize}
For quasi transitive graphs, these numbers satisfy the inequalities $p_c\leq p_{[2\to 2]}\leq p_{\rm exp}\leq p_u$,  ( \cite{Sharpness_2}, \cite[Theorem 1.1.2]{Sharpness}, \cite[Theorem 2.2]{Hutchcroft2_2})
We add two more specimen to this zoo of critical values. They implicitly depend on the random walk, we always choose the standard one.
\begin{defn}
\begin{enumerate}
    \item Let $p_{\rm Ram}$ (for Ramanujan) be the supremum of $p$ such that $\rho_{\cB_p}=\rho_{G}.$
    \item Let $p_{\rm cK}$ (for co-Kaimanovich) be the infimum of $p$ such that $\rho_{\cB_p}=1.$
\end{enumerate}
\end{defn}

The reason for the term ``co-Kaimanovich" is as follows. For spectral radius (as opposed to co-spectral radius) of relations $\cR$, having $\rho(\cR,\nu)=1$ for some countably supported $\nu\in \Prob([\cR])$ whose support generates $\cR$ is \emph{not} equivalent to hyperfiniteness of $\cR$ as discussed in detail in work of Kaimanovich \cite{VKLeafcounterexample}. This work of Kaimanovich greatly clarified a common misconception in the literature, and gave precise and tractable criteria to verify when a relation is hyperfinite in terms of the data of several spectral radii. E.g. this work  shows that spectral radius $1$ is equivalent to having $\rho(\cR,\nu)=1$ for \emph{every} countably supported $\nu\in \Prob([\cR])$ whose support generates $\cR$. It is thus reasonable to call a relation $\cR$ \emph{Kaimanovich} if there is a $\nu\in \Prob([\cR])$ whose support generates $\cR$ and has $\rho(\cR,\nu)=1$. Since co-spectral radius heuristically plays the role of spectral radius in a quotient, this motivates the term co-Kaimanovich.



We now compare how these critical values are related.

\begin{prop}
\begin{enumerate}
    \item  $p_{\rm Ram}\leq p_{\rm cK}$,
    \item $p_{\rm cK}\leq p_u$
    \item $p_{\rm exp}\leq p_{\rm cK}$,
    \item $p_{[2\to 2]}\leq p_{\rm Ram}$,
\end{enumerate}
\end{prop}

We remark that (4) is equivalent to a theorem of Hutchcroft \cite[Proposition 6.4]{Hutchcroft}, but as our proof is short we include the proof for completeness.

\begin{proof}
(1) If $G$ is amenable, then $1\geq \rho_{\cB_{p}}\geq \rho_{G}\geq 1$. 
If $G$ is non amenable, then $\rho_{G}<1$. The inequality follows now from the obvious monotonicity of $\rho_{\cB_p}$ in $p$.

(2) Let $p>p_{u}$. Let $\cS\leq \cR$ be the inclusion of equivalence relations constructed in Proposition \ref{prop:PercGraph}.  The infinite connected component of $\cB_{p}$ in the percolated graph corresponds to an $\cS$-invariant positive measure subset $E\subseteq \Omega_\#$ (namely $E=\{\omega:[\omega]_{\cS} \textnormal{ is infinite}\}$). Uniqueness of infinite connected component tells us that $\cR\big|_{E}=\cS\big|_{E}$. For $l\in \N$, we let $p_{l,x,E}$ be the probability that the random walk starting at $x$ is in $E$ after $l$ steps.
Since $\cR|_{E}=\cS|_{E}$, we have $p_{2n,x,\cS}\geq p_{2n,x,E}$. 
Thus, for almost every $x\in X$ we have 
\[\rho^{\cS}(x)\geq \rho^{\cR}_{E}(x)=\rho^{\cR}(x)1_{\widetilde{E}}(x)=1_{\widetilde{E}}(x),\]
where in the first equality we use Theorem \ref{thm: satuartion invariance}. Since $\mu(E)>0$ we obtain that 
\[\rho_{\cB_{p}}=\|\rho^{\cS}\|_{\infty}=1.\]


(3) Let \[\xi_p:=-\limsup_{n\to\infty} \frac{1}{n}\log \sup_{u,v}\{\tau_p(u,v)|\dist(u,v)\geq n\}.\] Then
$p_{\rm exp}=\sup \{p\in [0,1] | \xi_p>0\}.$ Let $p<p_{\rm exp}$. The spheres in $G$ grow at most exponentially fast so there exists $q>1$ such that
\[A^q:=\sum_{v\in G} |\tau_p(o,v)|^q< +\infty.\]
We have
\[ p_n(o,[o]_{\cB_p})=\sum_{v\in G} p_n(o,v)\tau_p(o,v).\]
By H\"older inequality
\[p_n(o,[o]_{\cB_p})\leq \|p_n(o,-)\|_{\frac{q}{q-1}}\|\tau_p(o,-)\|_q=A \|p_n(o,-)\|_{\frac{q}{q-1}}.\]
By the Riesz-Thorin theorem
\[\|p_n(o,-)\|_{\frac{q}{q-1}}\leq \|p_n(o,-)\|_1^{1-\frac{2}{q}}\|p_n(o,-)\|_2^{\frac{2}{q}}\leq \rho_G^{\frac{2n}{q}}.\]
Hence, $p_n(o,[o]_{\cB_p})\leq A \rho_G^{\frac{2n}{q}}$. The upper bound decays exponentially in $n$, so $\rho_{\cB_p}<1.$ This proves that $p_{\rm exp}\leq p_{\rm cK}.$

(4) Let $p<p_{[2\to 2]}$, so that $\|T_p\|_{L^2\to L^2}<\infty$. 
Fix the root $o\in V(G)$ and let $v\in V(G)$ be a vertex. Let $f_n(v):=p_n(o,v)$, i.e. the probability of going from $o$ to $v$ in time $n$. Note that $\|f_n\|_2^2=p_{2n}(o,o).$
We have
\[ T_p f_n(o)=\sum_{v\in V(G)} \tau_p(o,v) p_n(o,v)=\mathbb E(p_n(o,[o]_{\cB_p})).\]
The operator $T_p$ is bounded, so $\|T_p f_n\|_2\ll \|f_n\|_2$. We deduce that
\[ \rho_{\cB_p}=\lim_{n\to\infty} \mathbb E(p_{2n}(o,[o]_{\cB_p}))^{\frac{1}{2n}}\leq \lim_{n\to n} p_{4n}(o,o)^{\frac{1}{4n}}=\rho_G.\]
The reverse inequality $\rho_G\leq \rho_{\cB_p}$ is always trues so $\rho_{\cB_p}=\rho_G$. This demonstrates that $p\leq p_{\rm Ram}$. It follows that $p_{[2\to 2]}\leq p_{\rm Ram}$.
\end{proof}
\begin{example}
\begin{enumerate}
    \item Let $\mathcal T_d$ be the $d$-regular tree. We have $p_{\rm cK}=1$ and $p_{\rm Ram}=\frac{1}{2}.$
\end{enumerate}
\end{example}

\subsection{Walk growth}\label{sec: walk growth}

Let $(G,o)$ be a unimodular random graph with degree at most $d\in \N$. Let $\cM_{d}$ be the space of rooted graphs where each vertex has degree at most $d$, modulo isomorphism. This can be turned into a compact metric space with the distance
\[\rho([(G,o)],[(G',o')])=\inf\{2^{-l}:(B_{G}(o,l),o)\cong (B_{G'}(o',l),o')\}.\]
Recall that the distribution $\eta$ of $[(G,o)]$ is an probability measure on $\cM_{d}$ which is invariant under the rerooting equivalence relation
\[\{([(G,o)],[(G',o')]):G\cong G \textnormal{ as unrooted graphs}\}\]
(we remark that $\eta$ being rerooting invariant does not characterize unimodularity, but that will not cause an issue for us here).
Let $w_{n}(o)$ be the number of walks of length $n$ starting at $o$. 

Note that, by Proposition \ref{prop:PercGraph}, we can find a p.m.p. countable equivalence relation $(\Omega_{\#},\nu_{\#},\cR)$ with a generating graph $\Phi=(\phi_{i})_{i\in I}$ so that the distribution of $[(\cG,o)]$ is the law of the rooted graph $(\cG_{\omega},\omega)$ defined for $\omega\in \Omega_{\#}$ where the vertex set of $\cG_{\omega}$ is $[\omega]_{\cR}$ and the edge set $\{(\omega',\phi_{i}^{\pm 1}(\omega'):w'\in [\omega]_{\cR},i\in I\}$.
Define $f_{k}\colon \cR\to [0,+\infty)$ by declaring $f_{k}(x,y)$ to be the number of paths from $x$ to $y$ in $G_{x}$. It is direct to verify the hypotheses of Theorem \ref{thm:generalized 23} and thus 
\[\lim_{k\to\infty}\left(\sum_{y\in [x]_{\cR}}f_{k}(x,y)\right)^{1/k}\]
exists almost everywhere. The above sum has the same law as $w_{k}(o)$, and this proves that $\lim_{n\to\infty}\frac{1}{n}\log w_{n}(o)$ exists.

\begin{thm}
Fix notation as above, and suppose $\eta$ is ergodic under the rerooting equivalence relation. Define $A\in B(L^{2}(\cM_{d},\eta))$ by
\[(Af)([(G,o)])=\sum_{v\thicksim o}f([(G,v)]).\]
Then
\[\|A\|=\lim_{n\to\infty}w_{n}(o)^{1/n}\]
\end{thm}

\begin{proof}
Observe that if $o'$ is a vertex in the rooted graph $(G,o)$ and if $l=d_{G}(o,o')$, then
\[w_{n-l}(o')\leq w_{n}(o)\leq w_{n+l}(o') \textnormal{for all $n\in \N$}\]
and so
\[\lim_{n\to\infty}w_{n}(o')^{1/n}=\lim_{n\to\infty}w_{n}(o)^{1/n}.\]
Thus, by ergodicity, $\lim_{n\to\infty}w_{n}(o)^{1/n}$ is almost surely constant.
By Theorem \ref{thm:generalized 23} (\ref{item:integral formula for sup}) this constant equals
\begin{equation}\label{eqn: walk growth constant}
\lim_{n\to\infty}\E(w_{n}(o))^{1/n}.
\end{equation}
A direct calculation shows that 
$\ip{A^{k}1,1}=\E(w_{k}(o)).$
Note that 
\[\overline{\Span\{1_{E}:E\subseteq \cM_{d} \textnormal{ is measurable}\}}=L^{2}(\cM_{d},\eta),\]
and that 
$\ip{A^{k}1_{E},1_{E}}\leq \ip{A^{k}1,1}$
for all measurable sets $E$. Thus, the same argument as in  Theorem \ref{T:existence of the relative spectral radius} (\ref{I:averageversion subrelation}) using Lemma \ref{L:norm on dense set of vectors} shows that (\ref{eqn: walk growth constant}) is equal to $\|A\|$. 

\end{proof}





%

\end{document}